\documentclass{article}
\usepackage[cp1251]{inputenc}
\usepackage[russian]{babel}
\usepackage{amssymb,amsfonts,amsmath,amsthm,amscd}
\usepackage{graphicx}
\usepackage{enumerate}
\usepackage{soul}
\usepackage[matrix,arrow,curve]{xy}
\usepackage{longtable}
\usepackage{array}
\RequirePackage{russlh}
\RequirePackage{mathlh}

\newdimen\symskip
\newdimen\defskip
\newdimen\parind
\parind=\parindent
\newdimen\leftmarge
\newdimen\theoremshape
\topsep\defskip

\makeatletter
\newcommand*{\клей}{\nobreak\hskip\z@skip}

\renewcommand{\"}{''}
\renewcommand{\:}{\textup{:}}
\renewcommand{\~}{\textup{;}}
\DeclareRobustCommand*{\т}{~\textemdash{} }
\DeclareRobustCommand*{\д}{\клей\hbox{-}\клей}
\newcommand{\no}{}
\renewcommand{\@listI}{\settowidth\labelwidth{\labheadi{\no}}\listipar{\parind}{\labelwidth}}
\newcommand{\listivpar}{\topsep\defskip\partopsep0pt\parsep-\parskip\itemsep0.5\topsep}
\newcommand{\listipar}[2]{\rightmargin0pt\leftmargin#1\labelsep#1\advance\labelsep-#2\itemindent0pt\listivpar}
\renewcommand{\@listii}{\settowidth\labelwidth{\labheadii{\@roman{\no}}}\listiipar{\parind}{\labelwidth}}
\newcommand{\listiivpar}{\topsep0.5\defskip\partopsep0pt\parsep-\parskip\itemsep0.5\topsep}
\newcommand{\listiipar}[2]{\rightmargin0pt\leftmargin#1\labelsep#1\advance\labelsep-#2\itemindent0pt\listiivpar}
\def\thempfn{\ifcase\value{footnote}1\or *\or **\or ***\else\@ctrerr\fi}
\renewcommand\footnoterule{%
  \kern-3\p@
  \hrule\@width1in
  \kern2.6\p@}
\makeatother

\makeatletter

\renewcommand{\@biblabel}[1]{[#1]}
\renewenvironment{thebibliography}[1]
     {\renewcommand{\refname}{Литература}%
      \section*{\refname}%
      \@mkboth{\MakeUppercase\refname}{\MakeUppercase\refname}%
      \list{\@biblabel{\@arabic\c@enumiv}}%
           {\itemsep\baselineskip
            \leftmargin\parind
            \settowidth\labelwidth{\@biblabel{#1}}%
            \labelsep\parind\advance\labelsep-\labelwidth
            \@openbib@code
            \usecounter{enumiv}%
            \let\p@enumiv\@empty
            \renewcommand\theenumiv{\@arabic\c@enumiv}}%
      \sloppy
      \clubpenalty4000
      \@clubpenalty\clubpenalty
      \widowpenalty4000%
      \sfcode`\.\@m}
     {\def\@noitemerr
       {\@latex@warning{Empty `thebibliography' environment}}%
      \endlist}

\def\@maketitle{%
  \newpage
  \vskip1em%
  УДК \udk%
  \vskip1em%
  \begin{center}\bf%
  \let\footnote\thanks%
   {\Large\@author\par}%
   \vskip1.5em%
   {\LARGE\@title\par}%
   \vskip1em%
   {\large\@date}%
  \end{center}%
  \par
  \vskip1.5em}

\makeatother

\makeatletter

\renewcommand\sectionmark[1]{%
 \markright{%
  \ifnum \c@secnumdepth >\z@
   \thesection. \ %
  \fi
 #1}}%

\renewcommand{\section}{\@startsection{section}{1}{0pt}%
{5.5ex plus .5ex minus .2ex}{1.5ex plus .3ex}%
{\center\normalfont\Large\bfseries\sffamily}}
\renewcommand{\subsection}{\@startsection{subsection}{2}{0pt}%
{4.5ex plus .4ex minus .2ex}{0.75ex plus .2ex}%
{\center\normalfont\large\bfseries\sffamily}}
\renewcommand{\subsubsection}{\@startsection{subsubsection}{3}{0pt}%
{2.5ex plus .5ex minus .2ex}{1ex plus .2ex}%
{\center\normalfont\bfseries\scshape}}

\def\@postskip@{\hskip.5em\relax}

\def\postsection{.\@postskip@}

\def\postsubsection{.\@postskip@}

\def\postsubsubsection{.\@postskip@}

\def\postparagraph{.\@postskip@}

\def\postsubparagraph{.\@postskip@}

\def\@seccntformat#1{\csname pre#1\endcsname\csname the#1\endcsname\csname post#1\endcsname}

\makeatother

\renewcommand{\thesection}{\textup{\arabic{section}}}

\newcommand{\parr}{\par\addvspace{\defskip}}
\newcommand{\theo}[2]{\newtheorem{#1}{#2}[section]}
\newcommand{\deff}[2]{\newenvironment{#1}{\parr\textbf{#2.}}{\parr}}
\theo{cas}{Случай}
\theo{problem}{Проблема}
\theo{theorem}{Теорема}
\theo{lemma}{Лемма}
\theo{prop}{Предложение}
\theo{stm}{Утверждение}
\theo{imp}{Следствие}
\theo{ex}{Пример}
\deff{df}{Определение}
\deff{note}{Замечание}
\deff{denote}{Обозначение}
\deff{denotes}{Обозначения}
\deff{hint}{Указание}
\deff{answer}{Ответ\:}

\makeatletter
\def\@begintheorem#1#2[#3]{%
  \deferred@thm@head{\the\thm@headfont \thm@indent
    \@ifempty{#1}{\let\thmname\@gobble}{\let\thmname\@iden}%
    \@ifempty{#2}{\let\thmnumber\@gobble}{\let\thmnumber\@iden}%
    \@ifempty{#3}{\let\thmnote\@gobble}{\let\thmnote\@iden}%
    \thm@notefont{\bfseries\upshape}%
    \indent%
    \thm@swap\swappedhead\thmhead{#1}{#2}{#3}%
    \the\thm@headpunct
    \thmheadnl 
    \hskip\thm@headsep
  }%
  \ignorespaces}
\renewenvironment{proof}{\parr\pushQED{\qed}\normalfont$\square\quad$}{\popQED\@endpefalse\parr}

\makeatother

\newcommand{\labheadi}[1]{\textup{#1)}}
\newcommand{\labheadii}[1]{\textup{(#1)}}

\newenvironment{nums}[1]{\renewcommand{\no}{#1}\begin{enumerate}}{\end{enumerate}}
\newcommand{\eqn}[1]{\begin{equation}#1\end{equation}}
\newcommand{\equ}[1]{\begin{equation*}#1\end{equation*}}

\newcommand{\ml}[1]{\begin{multline*}#1\end{multline*}}

\newcommand{\case}[1]{\begin{cases}#1\end{cases}}
\newcommand{\rbmat}[1]{\begin{pmatrix}#1\end{pmatrix}}

\makeatletter

\def\LT@makecaption#1#2#3{%
  \LT@mcol\LT@cols c{\hbox to\z@{\hss\parbox[t]\LTcapwidth{%
    \sbox\@tempboxa{#1{#2. }#3}%
    \ifdim\wd\@tempboxa>\hsize
      #1{#2. }#3%
    \else
      \hbox to\hsize{\hfil\box\@tempboxa\hfil}%
    \fi
    \endgraf\vskip\baselineskip}%
  \hss}}}

\@addtoreset{equation}{section}
\@addtoreset{footnote}{section}

\makeatother

\newcounter{numt}
\newcounter{col}
\renewcommand{\thenumt}{\arabic{numt})}
\newcommand{\news}{\\\hline}
\newcommand{\refs}{\refstepcounter{numt}}
\newcommand{\n}[1]{\news\an{#1}}
\newcommand{\an}[1]{\refs\label{#1}\thenumt&}

\newcommand{\nc}{&\refstepcounter{col}}
\newcommand{\lont}[5]{%
\renewcommand{\tablename}{Table}
\setcounter{numt}{0}\setcounter{col}{1}
\begin{longtable}{|r|#1}\caption{#2}\label{#3}\\\hline №\nc#4\\\hline\endhead
\multicolumn{\value{col}}{c}{\textit{Продолж. на сл. стр.}}
\endfoot
\endlastfoot
#5\\\hline\end{longtable}}

\renewcommand{\ge}{\geqslant}
\renewcommand{\le}{\leqslant}
\newcommand{\fa}{\,\forall\,}

\newcommand{\exu}{\,\exists\,!\;}
\newcommand{\bes}{\infty}

\newcommand{\subs}{\subset}
\newcommand{\sups}{\supset}

\newcommand{\sm}{\setminus}
\newcommand{\cln}{\colon}
\newcommand{\nl}{\lhd}

\newcommand{\Lra}{\Leftrightarrow}

\newcommand{\hra}{\hookrightarrow}

\newcommand{\ol}{\overline}
\newcommand{\wt}{\widetilde}
\newcommand{\wh}{\widehat}

\newcommand{\suml}[2]{\sum\limits_{{#1}}^{{#2}}}
\newcommand{\sums}[1]{\sum\limits_{{#1}}}

\newcommand{\prodl}[2]{\prod\limits_{{#1}}^{{#2}}}

\newcommand*{\bw}[1]{#1\nobreak\discretionary{}{\hbox{$\mathsurround=0pt #1$}}{}}
\newcommand{\sco}{,\ldots,}

\newcommand{\sle}{\bw\le\ldots\bw\le}
\newcommand{\sge}{\bw\ge\ldots\bw\ge}

\newcommand{\ha}[1]{\left\langle#1\right\rangle}

\newcommand{\br}[1]{\bigl(#1\bigr)}
\newcommand{\Br}[1]{\Bigl(#1\Bigr)}
\newcommand{\bbr}[1]{\biggl(#1\biggr)}
\newcommand{\ter}[1]{\textup{(}#1\textup{)}}
\newcommand{\bm}[1]{\bigl|#1\bigr|}
\newcommand{\hn}[1]{\left\|#1\right\|}

\newcommand{\bs}[1]{\bigl[#1\bigr]}
\newcommand{\BS}[1]{\Bigl[#1\Bigr]}

\newcommand{\bc}[1]{\bigl\{#1\bigr\}}
\newcommand{\BC}[1]{\Bigl\{#1\Bigr\}}
\newcommand{\hb}[1]{\hskip-1pt\left/#1\right.\hskip-1pt}
\newcommand{\hbr}[1]{\left.\hskip-1pt#1\right/\hskip-1pt}

\newcommand{\mbb}{\mathbb}
\newcommand{\mbf}{\mathbf}
\newcommand{\mcl}{\mathcal}
\newcommand{\mfr}{\mathfrak}
\newcommand{\mrm}{\mathrm}
\newcommand{\R}{\mbb{R}}

\newcommand{\Z}{\mbb{Z}}
\newcommand{\N}{\mbb{N}}
\newcommand{\T}{\mbb{T}}
\newcommand{\F}{\mbb{F}}
\newcommand{\Cbb}{\mbb{C}}
\newcommand{\Hbb}{\mbb{H}}
\newcommand{\RP}{\mbb{R}\mrm{P}}
\newcommand{\CP}{\mbb{C}\mrm{P}}

\newcommand{\Zc}{\mcl{Z}}
\newcommand{\ggt}{\mfr{g}}

\newcommand{\sug}{\mfr{su}}
\newcommand{\pd}{\partial}

\newcommand{\al}{\alpha}
\newcommand{\be}{\beta}
\newcommand{\ga}{\gamma}

\newcommand{\de}{\delta}
\newcommand{\De}{\Delta}
\newcommand{\ep}{\varepsilon}
\newcommand{\la}{\lambda}

\newcommand{\nab}{\nabla}
\newcommand{\rh}{\rho}
\newcommand{\si}{\sigma}
\newcommand{\ta}{\theta}
\newcommand{\ph}{\varphi}
\newcommand{\om}{\omega}
\newcommand{\Om}{\Omega}

\DeclareMathOperator{\Lie}{Lie}
\DeclareMathOperator{\Ker}{Ker}
\DeclareMathOperator{\Ad}{Ad}
\DeclareMathOperator{\Mat}{Mat}
\DeclareMathOperator{\End}{End}
\DeclareMathOperator{\Arg}{Arg}
\DeclareMathOperator{\Int}{Int}
\DeclareMathOperator{\Rea}{Re}
\DeclareMathOperator{\Img}{Im}
\DeclareMathOperator{\rk}{rk}
\DeclareMathOperator{\tr}{tr}
\DeclareMathOperator{\diag}{diag}
\DeclareMathOperator{\id}{id}

\newcommand{\GL}{\mbf{GL}}
\newcommand{\SL}{\mbf{SL}}
\newcommand{\Or}{\mbf{O}}

\newcommand{\SO}{\mbf{SO}}
\newcommand{\SU}{\mbf{SU}}

\newcommand{\thra}{\twoheadrightarrow}

\newcommand{\lp}[1]{\llap{$#1$}}

\begin{document}

\author{О. Г. Стырт}
\title{О пространстве орбит\\
трёхмерной компактной линейной группы Ли}
\date{}
\newcommand{\udk}{512.815.1, 512.816.2}

\maketitle

{\leftskip\parind\rightskip\parind
Исследуется вопрос о том, является ли топологический фактор вещественного линейного представления простой трёхмерной компактной группы Ли многообразием. Получена верхняя оценка размерности представления, фактор которого является многообразием; разобрано большинство оставшихся случаев.

Библиография: 1 наименование.

\smallskip

\textbf{Ключевые слова\:} группа Ли, топологический фактор действия.\par}

\section{Введение}\label{introd}

Настоящая работа является непосредственным продолжением статьи~\cite{comm}. Прежде всего дадим три базовых определения, игравших ключевую роль
и~в~\cite{comm}.

\begin{df} Непрерывное отображение гладких многообразий назовём \textit{ку\-соч\-но-глад\-ким}, если оно переводит любое гладкое подмногообразие в~конечное объединение гладких подмногообразий.
\end{df}

В частности, всякое собственное гладкое отображение гладких многообразий является ку\-соч\-но-глад\-ким.

Рассмотрим дифференцируемое действие некоторой компактной группы Ли~$G$ на гладком многообразии~$M$.

\begin{df} Будем говорить, что фактор действия $G\cln M$ \textit{диффеоморфен} (\textit{ку\-соч\-но-диф\-фео\-мор\-фен}) гладкому многообразию~$M'$, если топологический фактор $M/G$ гомеоморфен~$M'$, причём отображение факторизации $M\to M'$ гладкое (ку\-соч\-но-гладкое).
\end{df}

\begin{df} Будем говорить, что фактор действия $G\cln M$ является \textit{гладким многообразием}, если он ку\-соч\-но-диф\-фео\-мор\-фен некоторому гладкому многообразию.
\end{df}

Перейдём непосредственно к~постановке задачи.

Рассмотрим линейное представление компактной группы Ли~$G$ в~вещественном пространстве~$V$. Нас по-прежнему (как и~в~\cite{comm}) интересует вопрос о~том, является ли фактор $V/G$ этого действия топологическим многообразием, а~также является ли он гладким многообразием. Следуя~\cite{comm}, будем далее для краткости называть топологическое многообразие просто <<многообразием>>.

Через~$G^0$ будем обозначать связную компоненту единицы группы~$G$, а~через~$\ggt$\т её касательную алгебру.

Случай, когда группа~$G^0$ коммутативна, был разобран в~\cite{comm}. Данная же работа посвящена исследованию поставленной проблемы в~предположении, что $\ggt\cong\sug_2$.

В пространстве~$V$ можно зафиксировать $G$\д инвариантное скалярное умножение. Тогда группа~$G$ действует ортогональными операторами\: $G\subs\Or(V)$.

Для любой одномерной подалгебры (что то же самое, одномерного подпространства) $\ggt'\subs\ggt$ подмножество $\ggt'V\subs V$ является, очевидно, подпространством.

Допустим, что $\ggt\cong\sug_2$\т что равносильно, группа~$G^0$ изоморфна одной из групп $\SU_2$ и~$\SO_3$.

Обозначим через $n_1\sco n_L$ размерности неприводимых компонент представления $\ggt\cln V$ (с учётом кратностей). Если все числа~$n_i$ равны~$1$, то группа~$G^0$ действует на~$V$ тождественно, и~вопрос описания фактора $V/G$ сводится к~аналогичному вопросу для действия конечной группы $G/G^0$ в~пространстве~$V$. Поэтому будем считать, что $n_1\sge n_l>1=n_{l+1}=\dots=n_L$, $l=1\sco N$. Число $\bs{\frac{n_i}{2}}$ является натуральным при $n_i>1$ и~равно нулю при $n_i=1$. Положим $q(V):=\suml{i=1}{L}\bs{\frac{n_i}{2}}=\suml{i=1}{l}\bs{\frac{n_i}{2}}\in\N$.

В~\S\,\ref{promain} будут доказаны теоремы~\ref{main}---\ref{main2}.

\begin{theorem}\label{main} Если $\ggt\cong\sug_2$, а~$V/G$\т гладкое многообразие, то $q(V)\le4$.
\end{theorem}

\begin{theorem}\label{main1} Если $\ggt\cong\sug_2$, а~$V/G$\т многообразие, то $q(V)>2$.
\end{theorem}

\begin{imp}\label{q34} Если $\ggt\cong\sug_2$, а~$V/G$\т гладкое многообразие, то $q(V)\in\{3;4\}$.
\end{imp}

\begin{theorem}\label{main2} Если $\ggt\cong\sug_2$, $G=G^0$, $V/G$\т гладкое многообразие, а~среди чисел $\bs{\frac{n_i}{2}}$, $i=1\sco l$, хотя бы одно нечётно, то $q(V)=3$.
\end{theorem}

Согласно следствию~\ref{q34} и~теореме~\ref{main2}, если $G=G^0$, а~$V/G$\т гладкое многообразие, то представление $G\cln V_0^{\perp}$
относится к~одному из типов, приведённых в~таблице~\ref{tabl}. В~\S\,\ref{part} мы докажем теоремы \ref{quat}---\ref{n53}, описывающие б\'ольшую часть этих случаев.

\lont{>{$}r<{$}|>{$}r<{$}|}{}{tabl}{l\nc n_1\sco n_l}{%
\an{cass44}
2 & 4,\,4
\n{cass43}
2 & 4,\,3
\n{cass333}
3 & 3,\,3,\,3
\n{cass54}
2 & 5,\,4
\n{cass7}
1 & 7
\n{cass8}
1 & 8
\n{cass9}
1 & 9
\n{cass53}
2 & 5,\,3
\n{cass55}
2 & 5,\,5}

\begin{theorem}\label{quat} Предположим, что $\ggt\cong\sug_2$, а~для чисел $l$ и~$n_1\sco n_l$ имеет место одна из следующих комбинаций\:
\begin{nums}{9}
\item\label{case44} $l=2$, $n_1=n_2=4$\~
\item\label{case43} $l=2$, $n_1=4$, $n_2=3$\~
\item\label{case333} $l=3$, $n_1=n_2=n_3=3$.
\end{nums}
Тогда фактор $V/G^0$ диффеоморфен векторному пространству, причём группа $G/G^0$ действует на нём линейно.
\end{theorem}

\begin{note} Теорема~\ref{quat} позволяет в~каждом из случаев~\ref{case44}---\ref{case333} её формулировки свести исходное представление $G\cln V$ к~линейному представлению конечной группы $G/G^0$ в~векторном пространстве~$V/G^0$ (см. лемму~\ref{red}).
\end{note}

\begin{theorem}\label{n54} Если $\ggt\cong\sug_2$, $G=G^0$, $l=2$, $n_1=5$, $n_2=4$, то фактор $V/G$ не является гладким многообразием.
\end{theorem}

\begin{theorem}\label{n7} Если $\ggt\cong\sug_2$, $G=G^0$, $l=1$, $n_1=7$, то $V/G\cong\R^4$.
\end{theorem}

\begin{theorem}\label{n8} Если $\ggt\cong\sug_2$, $G=G^0$, $l=1$, $n_1=8$, то $V/G\cong\R^5$.
\end{theorem}

\begin{theorem}\label{n53} Если $G=G^0$, $l=2$, $n_1=5$, $n_2=3$, то $V/G\cong\R^5$.
\end{theorem}

\section{Обозначения и~вспомогательные факты}

Здесь мы напомним некоторые принятые в~\cite{comm} определения и~обозначения. Кроме того, будут сформулированы несколько результатов (леммы~\ref{red}---\ref{abel}), полученных в~\cite{comm}, и~их простейших следствий.

\begin{lemma}\label{red} Пусть имеется линейное представление компактной группы Ли~$G$ в~векторном пространстве~$V$. Предположим, что для некоторой
нормальной подгруппы $H\nl G$ фактор $V/H$ диффеоморфен гладкому многообразию~$M$. Тогда факторы $M/G$ и~$V/G$ являются или не являются гладкими
многообразиями одновременно.
\end{lemma} 

\begin{lemma}\label{stratk} Допустим, что у~орбиты общего положения представления $G\cln V$ гомотопическая группа $\pi_k$ \ter{$k>0$} нетривиальна, а~в~$V/G$ любой страт, отличный от главного, имеет коразмерность более $k+2$. Тогда $V/G$ не есть гладкое многообразие.
\end{lemma}

Пусть $G_v$\т стабилизатор вектора $v\in V$, а~$\ggt_v:=\Lie G_v=\{\xi\in\ggt\cln\xi v=0\}$. Группа~$G_v$ переводит в~себя подпространства $T_v(Gv)=\ggt v$ и~$N_v:=(\ggt v)^{\perp}$. Через~$M_v$ будем обозначать ортогональное дополнение в~$N_v$ к~подпространству $N_v^{G_v}$ неподвижных векторов для действия $G_v\cln N_v$. Тогда $V=\ggt v\oplus N_v^{G_v}\oplus M_v$ и~$G_vM_v=M_v$.

\begin{lemma}\label{slice} Если $v\in V$\т произвольный вектор, а~$V/G$\т \ter{гладкое} многообразие, то и~$N_v/G_v$\т \ter{гладкое} многообразие.
\end{lemma}

\begin{lemma}\label{MvV'} Любое $G^0$\д инвариантное подпространство $V'\subs V$ содержит вектор~$v$, для которого $M_v\perp V'$.
\end{lemma} 

Для $g\in G$ положим $\om(g):=\rk(E-g)-\rk\br{E-\Ad(g)}$. Далее, пусть $\Om:=\bc{g\bw\in G\cln\om(g)\bw\in\{0;2\}}\bw\subs G$ и
\eqn{\label{V0}
V_0:=\bc{v\in V\cln G^0v=\{v\}}=\{v\in V\cln\ggt v=0\}\subs V.}

Для всякого вектора $v\in V$ с~конечным стабилизатором и~элемента $g\in G_v$ имеем $\dim\br{(E-g)N_v}=\om(g)$.

\begin{lemma}\label{refstab} Если стабилизатор~$G_v$ вектора $v\in V$ конечен, а~фактор $V/G$ является гладким многообразием, то $G_v=\ha{G_v\cap\Om}$.
\end{lemma}

Для подпространства $W\subs V$ обозначим через $G[W]\subs G$ подгруппу Ли всех элементов из~$G$, действующих тождественно на~$W^{\perp}$. Очевидно, что подпространство~$W$ инвариантно относительно $G[W]$.

\begin{lemma}\label{transfer} Пусть $W\subs V$\т ненулевое $G^0$\д инвариантное подпространство, причём на его единичной сфере группа $G[W]$ действует транзитивно. Тогда $V/G$\т не многообразие.
\end{lemma}

Пусть $P$\т множество векторов в~конечномерном пространстве над произвольным полем. Некоторые из них могут совпадать\~ кратности учитываются.
Количество ненулевых векторов множества~$P$ (с учётом кратностей) обозначим через $\hn{P}$. Напомним определения \textit{$m$\д устойчивых} ($m\in\N$) и~\textit{неразложимых} множеств векторов, данные в~\cite{comm} и~необходимые также в~данной работе.


Разложением множества векторов конечномерного линейного пространства на компоненты будем называть его представление в~виде объединения своих подмножеств, линейные оболочки которых линейно независимы. Если среди этих линейных оболочек по крайней мере две нетривиальны, то такое разложение назовём \textit{собственным}. Будем говорить, что множество \textit{неразложимо}, если оно не допускает ни одного собственного разложения на компоненты. Всякое множество векторов разлагается на неразложимые компоненты единственным образом (с точностью до распределения нулевого вектора), причём для любого его разложения на компоненты каждая компонента является объединением некоторых его неразложимых компонент (вновь с~точностью до нулевого вектора).

\begin{df} Конечное множество векторов конечномерного пространства, рассматриваемое с~учётом кратностей своих
элементов, назовём \textit{$q$\д устойчивым} ($q\bw\ge0$), если его линейная оболочка не меняется при удалении из него любых векторов в~количестве не более~$q$ (с~учётом кратностей).
\end{df}

Теперь будем считать группу~$G^0$ коммутативной. Любое её неприводимое представление либо одномерно, либо двумерно. Во втором случае оно обладает $G^0$\д инвариантной комплексной структурой, и~ему соответствует вес $\la\cln G^0\to\T$, который можно отождествить с~его дифференциалом\т линейной функцией $\la\cln\ggt\to\R$. Последнюю можно понимать как вектор из~$\ggt$, используя скалярное умножение на~$\ggt$, инвариантное относительно $\Ad(G)$. Оператор $G^0$\д инвариантной комплексной структуры данного представления определён с~точностью до смены знака\~ то же можно сказать и~про вес $\la\in\ggt$. Одномерному представлению~$G^0$ сопоставим нулевой вес $\la\in\ggt$. Через $P\subs\ggt$ обозначим множество весов~$\la$, соответствующее разложению~$V$ в~прямую сумму неприводимых представлений группы~$G^0$ (с учётом кратностей). Это множество не зависит от выбора разложения (с точностью до знаков весов). Подпространство~$V_0^{\perp}$ имеет $G^0$\д инвариантную структуру комплексного пространства размерности $\hn{P}$. Пусть $V_{\la}$ ($\la\in P$)\т изотипная компонента представления $G^0\cln V$, соответствующая неприводимым представлениям с~весом~$\la$, а~$V_Q$ ($Q\subs P$)\т прямая сумма всех компонент $V_{\la}$, $\la\in Q$. В~частности, такое определение~$V_0$ согласуется с~определением~\eqref{V0}, где не требуется коммутативность~$G^0$.

\begin{lemma} Если $V/G$\т многообразие, то $P$\т $1$\д устойчивое множество.
\end{lemma} 

\begin{imp}\label{1sta} Если $V/G$\т многообразие, то $\hn{P}\ne1$.
\end{imp}

Множество всех ненулевых весов в~$P$ можно разложить в~дизъюнктное объединение неразложимых компонент $Q\subs P\sm\{0\}$.

Для произвольного линейного представления $G\cln V$ и~целого неотрицательного числа~$d$ будем, если не оговорено противное, подразумевать под записью $(V\oplus\R^d)/G$ фактор линейного представления
\equ{
G\cln V\oplus\R^d,\;g\cln v+x\to gv+x,\;g\in G,\,v\in V,\,x\in\R^d.}

\begin{lemma}\label{irred} Допустим, что $P$\т $2$\д устойчивое множество, а~$V/G$\т гладкое многообразие. Тогда для любой неразложимой компоненты
$Q\subs P\sm\{0\}$ найдётся число $d\ge0$, такое что $(V_Q\oplus\R^d)/G$\т гладкое многообразие.
\end{lemma} 

Предположим, что $\dim G=1$ (как следствие, алгебра~$\ggt$ и~группа~$G^0$ одномерны и~коммутативны). В~этом случае $P$\т неразложимое множество, а~его $2$\д устойчивость эквивалентна тому, что $\hn{P}\notin\{1;2\}$.

\begin{lemma}\label{1dim} Если $P$\т $2$\д устойчивое множество, $\dim G=1$, $V_0=0\ne V$, а~фактор $(V\oplus\R^d)/G$ является гладким многообразием
для некоторого~$d$, то $\hn{P}=3$.
\end{lemma} 

\begin{lemma}\label{abel} Если $P$\т $2$\д устойчивое множество, $\dim G=1$, $\Ad(G)=\{E\}$ и~$V_0\bw\ne V$, то $V/G$ не есть гладкое многообразие.
\end{lemma}

\begin{imp}\label{1dim3} Допустим, что $\dim G=1$, а~$V/G$\т гладкое многообразие. Тогда
\begin{nums}{9}
\item число $\hn{P}$ не превосходит~$3$\~
\item если $\Ad(G)=\{E\}$, то $\hn{P}\le2$.
\end{nums}
\end{imp}

\begin{proof} Если $\hn{P}<3$, то оба утверждения верны.

Далее будем считать, что $\hn{P}\ge3$. В~таком случае $V_0\ne V$, а~$P$\т $2$\д устойчивое множество. Теперь первое утверждение вытекает из лемм~\ref{irred} и~\ref{1dim} (при $Q:=P\sm\{0\}$), а~второе\т из леммы~\ref{abel}.
\end{proof}

\begin{imp}\label{xiV} Предположим, что $\dim G=1$. Тогда
\begin{nums}{9}
\item если $V/G$\т многообразие, то $\dim(\ggt V)\ne2$\~
\item если $V/G$\т гладкое многообразие, то $\dim(\ggt V)\le6$\~
\item если $\Ad(G)=\{E\}$, а~$V/G$\т гладкое многообразие, то $\dim(\ggt V)\le4$.
\end{nums}
\end{imp}

\begin{proof} Имеем $\dim(\ggt V)=\dim V_0^{\perp}=2\dim_{\Cbb}V_0^{\perp}=2\hn{P}$. Осталось воспользоваться следствиями~\ref{1sta} и~\ref{1dim3}.
\end{proof}

\section{Представления трёхмерных групп}\label{su2}


Здесь будут перечислены основные свойства неприводимых вещественных представлений групп $\SU_2$ и~$\SO_3$.

Пусть $V(m)$ ($m\in\N$)\т комплексное представление группы $\SU_2$ на пространстве однородных многочленов из $\Cbb[x,y]$ степени $m-1$ сдвигом аргумента.

Размерность каждого неприводимого представления группы $\SU_2$ либо кратна~$4$, либо нечётна. Неприводимое представление размерности $4m$\т это овеществление комплексного представления $V(2m)$. Неприводимое представление размерности $2m+1$\т это вещественная форма $V_{\R}(2m+1)$ комплексного представления $V(2m+1)$, состоящая из всех многочленов $\sums{|k|\le m}c_kx^{m+k}y^{m-k}$, таких что $c_{-k}=(-1)^k\,\ol{c_k}$. Любое неприводимое представление нечётной размерности абсолютно неприводимо, а~тело его эндоморфизмов есть $\{c E\cln c\in\R\}$.

Всякое неприводимое представление группы $\SO_3$ имеет нечётную размерность и~получается из неприводимого представления $\SU_2$ той же размерности факторизацией $\SU_2\thra\SU_2\hb{\{\pm E\}}\cong\SO_3$ по его ядру неэффективности $\{\pm E\}\subs\SU_2$.

\begin{lemma}\label{odd} Рассмотрим $(2m+1)$\д мерное неприводимое представление группы, изоморфной $\SU_2$ либо~$\SO_3$.
\begin{nums}{9}
\item Если число~$m$ нечётно, то существует вектор, стабилизатор которого является одномерным тором.
\item Если $m>1$, то найдётся вектор с~нетривиальным конечным стабилизатором.
\end{nums}
\end{lemma}

\begin{proof} Можно отождествить данное представление с~$V_{\R}(2m+1)$. Положим $f:=x^my^m\in V_{\R}(2m+1)$.

\begin{nums}{9}
\item Подгруппа всех элементов из $\SU_2$, переводящих в~себя прямую $\R f$, совпадает с~подгруппой всех диагональных и~побочно-диагональных матриц.
Все диагональные матрицы из $\SU_2$ переводят~$f$ в~себя, а~побочно-диагональные\т в~многочлен $(-1)^mf$. Таким образом, $f$\т искомый многочлен в~первом утверждении.

\item По условию $m>1$. Всякий многочлен из $V_{\R}(2m+1)$ с~бесконечным стабилизатором в~группе $\SU_2$ лежит в~одной орбите с~некоторым многочленом
из $\R f$ и~потому имеет кратные корни. Следовательно, многочлен $x^{2m}+(-1)^my^{2m}\in V_{\R}(2m+1)$ является требуемым\: он не имеет кратных корней и~переходит в~себя под действием элемента $\diag\br{e^{\pi i/m};e^{-\pi i/m}}\in\SU_2\sm\{\pm E\}$.\qedhere
\end{nums}
\end{proof}

Зачастую оказывается удобным отождествлять группу $\SU_2$ с~группой $\bc{\la\bw\in\Hbb\cln|\la|\bw=1}$ по умножению. Её трёхмерное неприводимое представление можно понимать как тавтологическое представление $\SO_3\cln\R^3$ или как действие группы кватернионов с~модулем~$1$ на пространстве~$\Hbb_0$ чисто мнимых кватернионов сопряжениями. Четырёхмерное же неприводимое представление группы $\SU_2$ есть не что иное как действие группы кватернионов с~модулем~$1$ на пространстве кватернионов правыми сдвигами, а~тело его эндоморфизмов есть тело операторов левых сдвигов на всевозможные кватернионы. В~обоих случаях действие группы на единичной сфере пространства представления транзитивно. Наиболее наглядной интерпретацией пятимерного неприводимого представления группы $\SO_3$ является её действие на пространстве симметрических матриц из $\Mat_{3\times3}(\R)$ со следом~$0$ матричными сопряжениями.

Как известно, при стереографической проекции $\CP^1\to S^2$ проективизация тавтологического представления $\SU_2\cln\Cbb^2$ становится тавтологическим действием $\SO_3\cln S^2$. Разлагая на линейные множители все ненулевые многочлены комплексного пространства $V(m)$ ($m\in\N$), можно отождествить множество его прямых с~симметрической степенью $SP^{m-1}(\CP^1)$. При этом проективизация комплексного представления $\SU_2\cln V(m)$ совпадёт с~тавтологическим действием $\SO_3\cln SP^{m-1}(S^2)$, поскольку при действии элемента $g\in\SU_2$ на проективном пространстве $PV(m)$ вектор $\rbmat{a\\b}\in\Cbb^2\sm\{0\}$, соответствующий линейному множителю $\ol{a}x+\ol{b}y$, умножается слева на матрицу~$g$. Допустим, что число~$m$ нечётно. Тогда разложение на линейные множители любого ненулевого многочлена $f\in V_{\R}(m)$ вместе с~множителем $ax+by$ включает в~себя и~$(-\ol{b}x+\ol{a}y)$\т что равносильно, соответствующий неупорядоченный набор из $m-1$ точки сферы~$S^2$ сохраняется при замене каждой точки на диаметрально противоположную ей. Если многочлен $f\in V(m)\sm\{0\}$ удовлетворяет импликации $(ax+by)\mid f\Lra(-\ol{b}x+\ol{a}y)\mid f$ для любых $a,b\in\Cbb$, то все числа $\la\in\Cbb$, такие что $\la f\in V_{\R}(m)$, образуют прямую в~$\Cbb$. Следовательно, проективизация вещественного представления $\SO_3\cln V_{\R}(m)$\т это то же самое, что тавтологическое действие $\SO_3\cln SP^{\frac{m-1}{2}}(\RP^2)$.

\section{Доказательства основных результатов}\label{promain}

Этот параграф будет посвящён доказательству теорем~\ref{main}---\ref{main2}.

Далее (до конца работы) будем считать, что $\ggt\cong\sug_2$, то есть что группа~$G^0$ изоморфна $\SU_2$ либо $\SO_3$. В~обозначениях и~соглашениях \S\,\ref{introd}, $L-l=\dim V_0$, $V_0\ne V$, а~числа $n_1\sco n_l$ суть размерности неприводимых компонент представления $\ggt\cln V_0^{\perp}$
(с~учётом кратностей), причём каждое из них либо кратно~$4$, либо нечётно.

Любая собственная подалгебра $\ggt'\subs\ggt$ одномерна, а~подпространство $\ggt'V$ имеет размерность $2q(V)$. Если $\xi\in\ggt\sm\{0\}$, то
$\xi V_0=0$, $\xi V=\xi V_0^{\perp}\subs V_0^{\perp}$, $\dim\br{\xi V_0^{\perp}}\bw=\dim(\xi V)\bw=2q(V)$, и
\eqn{\label{qV}
\dim\Br{\Ker\br{\xi|_{V_0^{\perp}}}}=\dim V_0^{\perp}-2q(V)=\suml{i=1}{l}n_i-2\suml{i=1}{l}\BS{\frac{n_i}{2}}=2\suml{i=1}{l}\BC{\frac{n_i}{2}}.}

\subsection{Доказательство теоремы~\ref{main}}

Гомотопическая группа $\pi_3(G)$ нетривиальна, так как существует накрытие $\SU_2\bw\thra\SO_3$, а~группа $\SU_2$ гомеоморфна трёхмерной сфере.

Достаточно доказать теорему при дополнительном предположении, что группа~$G$ действует на~$V_0$ тождественно. В~самом деле, найдётся вектор
$v\in V_0$, для которого $M_v\subs V_0^{\perp}$ (см. лемму~\ref{MvV'}). Тогда $\ggt v=0$, $\Lie G_v=\ggt_v=\ggt$, $N_v=V$, $V^{G_v}=N_v^{G_v}=M_v^{\perp}\sups V_0$. Если $V/G$\т гладкое многообразие, то и~$V/G_v=N_v/G_v$\т гладкое многообразие. При переходе от представления $G\cln V$ к~представлению $G_v\cln V$
остаются прежними представление алгебры $\ggt=\ggt_v$ в~пространстве~$V$, подпространство~$V_0$ и~число $q(V)$. Наконец, группа~$G_v$ действует тождественно на~$V_0$.

В дальнейшем будем считать, что $V/G$\т гладкое многообразие, а~$G$ действует на~$V_0$ тождественно. Требуется доказать, что $q(V)\le4$.

Допустим, что существует вектор $v\in V$ с~одномерным стабилизатором в~$G$. Согласно лемме~\ref{slice}, $N_v/G_v$\т гладкое многообразие. В~силу следствия~\ref{xiV}, $\dim(\ggt_v N_v)\le6$. Кроме того, $\dim(\ggt v)=2$ и~$\ggt_v V=\ggt_v(\ggt v)\oplus\ggt_v N_v$, поэтому $\dim(\ggt_v V)\le8$, $q(V)\bw=\frac{1}{2}\dim(\ggt_v V)\bw\le4$.

Теперь предположим, что в~пространстве~$V$ нет ни одного вектора с~одномерным стабилизатором.

Для любого вектора $v\in V\sm V_0$ алгебра~$\ggt_v$ не является одномерной и~не совпадает с~$\ggt$, значит, $\ggt_v=0$, $\dim(\ggt v)=3$. Отсюда всякое $\ggt$\д инвариантное подпространство размерности не более~$2$ содержится в~$V_0$.

Допустим, что $q(V)>4$.

Покажем, что $\Om=\{E\}$, то есть что произвольный оператор $g\in\Om$ тождественный.

Если $\Ad(g)=E$, то оператор~$g$ коммутирует со всеми операторами из~$\ggt$, следовательно, подпространство $(E-g)V$ инвариантно относительно~$\ggt$. Оно имеет размерность $\rk(E-g)=\om(g)\le2$ и~потому содержится в~$V_0$. Отсюда $(E-g)V=(E-g)V_0=0$, $g=E$.

Если же $\Ad(g)\ne E$, то $\rk\br{E-\Ad(g)}=2$, и~$\rk(E-g)=\om(g)+2\le4$. Рассмотрим вектор $\xi\in\ggt$, для которого $\eta:=\br{E-\Ad(g)}\xi\ne0$. Тогда
\equ{
\eta V^g=\br{\eta+\Ad(g)\xi(E-g)}V^g=(E-g)(\xi V^g).}
Значит, каждое из подпространств~$V^g$ и~$(E-g)V$ под действием~$\eta$ переходит в
подпространство размерности не более $\rk(E-g)\le4$. Для оператора $g\in\Or(V)$ выполнено равенство $V=V^g\oplus(E-g)V$, откуда $\dim(\eta V)\le8$, $q(V)=\frac{1}{2}\dim(\eta V)\le4$. Получили противоречие.

Тем самым мы доказали, что $\Om=\{E\}$.

Согласно~\eqref{qV}, $\dim V_0^{\perp}\ge2q(V)>8$, $\dim V>\dim V_0+8$, 
$\dim(V/G)\bw\ge\dim V-\dim G\bw=\dim V-3\bw>\dim V_0+5$. Для любого вектора $v\in V\sm V_0$ алгебра~$\ggt_v$ тривиальна, то есть стабилизатор~$G_v$ конечен. В~силу леммы~\ref{refstab}, он порождён своим пересечением с~$\Om$ и~поэтому тривиален. Следовательно, все векторы из $V\sm V_0$ лежат в~главном страте, стабилизатор общего положения тривиален, орбита общего положения гомеоморфна~$G$, а~её гомотопическая группа~$\pi_3$ нетривиальна. Всякий страт в~$V$, отличный от главного, содержится в~$V_0$. Значит, все страты фактора $V/G$, кроме главного, имеют размерность не более $\dim(V_0/G)=\dim V_0$ и~коразмерность не менее $\dim(V/G)-\dim V_0>5$. Применяя лемму~\ref{stratk} для числа $k:=3$, приходим к~противоречию с~тем, что $V/G$\т гладкое многообразие.

Теорема~\ref{main} доказана.

\subsection{Доказательство теоремы~\ref{main1}}

Допустим, что $q(V)\le2$.

\begin{nums}{9}
\item\label{so3} Если $q(V)=1$, то $l=1$ и~$\bs{\frac{n_1}{2}}=1$, $n_1=3$.
\item\label{quatr} Если пространство~$V$ не содержит ни одного вектора с~одномерным стабилизатором в~$G$, то ограничение любого оператора
$\xi\in\ggt\sm\{0\}$ на подпространство~$V_0^{\perp}$ невырождено. В~силу~\eqref{qV}, числа $n_1\sco n_l$ кратны~$4$, а~также
$\suml{i=1}{l}n_i=\dim V_0^{\perp}=2q(V)\le4$, откуда $l=1$ и~$n_1=4$.
\end{nums}

В~каждом из случаев \ref{so3} и~\ref{quatr} представление $G^0\cln V_0^{\perp}$ неприводимо, его размерность равна $3$ либо~$4$, и~на его единичной сфере группа $G^0\subs G\bs{V_0^{\perp}}$ действует транзитивно. Пользуясь леммой~\ref{transfer}, приходим к~противоречию с~тем, что $V/G$\т многообразие.

Значит, $q(V)=2$, причём существует вектор $v\in V$ с~одномерным стабилизатором~$G_v$. Согласно лемме~\ref{slice}, $N_v/G_v$\т многообразие. В~силу следствия~\ref{xiV}, $\dim(\ggt_v N_v)\bw\ne2$. С~другой стороны, $\dim\br{\ggt_v(\ggt v)}=\dim(\ggt v)=2$ и~$\ggt_v V=\ggt_v(\ggt v)\oplus\ggt_v N_v$, поэтому $\dim(\ggt_v N_v)=2q(V)-2=2$. Полученное противоречие доказывает теорему~\ref{main1}.

Из теорем~\ref{main} и~\ref{main1} сразу вытекает следствие~\ref{q34}.

\subsection{Доказательство теоремы~\ref{main2}}

Предположим, что группа~$G$ изоморфна одной из групп $\SU_2$ и~$\SO_3$, фактор $V/G$ является гладким многообразием, а~число $m:=\bs{\frac{n_i}{2}}$ нечётно для некоторого $i\in\{1\sco l\}$. Тогда $n_i=2m+1$, поскольку $2m$\т чётное число, не кратное~$4$. Одна из неприводимых компонент представления $\ggt\cln V_0^{\perp}$ имеет размерность $2m+1$ и~по лемме~\ref{odd} содержит вектор~$v$ с~одномерным коммутативным стабилизатором. Согласно лемме~\ref{slice}, $N_v/G_v$\т гладкое многообразие. В~силу следствия~\ref{xiV}, $\dim(\ggt_v N_v)\le4$. Кроме того, $\dim(\ggt v)=2$ и~$\ggt_v V=\ggt_v(\ggt v)\oplus\ggt_v N_v$, откуда $\dim(\ggt_v V)\le6$, $q(V)=\frac{1}{2}\dim(\ggt_v V)\le3$. Осталось воспользоваться следствием~\ref{q34}.

Тем самым теорема~\ref{main2} доказана.

\section{Разбор частных случаев}\label{part}

Здесь будут доказаны теоремы~\ref{quat}---\ref{n53}.

\subsection{Доказательство теоремы~\ref{quat}}

Все автоморфизмы алгебры $\ggt\cong\sug_2$ внутренние, и, значит, подгруппа $\Zc(G^0)\bw=\bc{g\in G\cln\Ad(g)=E}$ пересекает все связные компоненты группы~$G$. Поэтому достаточно доказать, что фактор действия $G^0\cln V$ диффеоморфен векторному пространству, причём на нём действует линейно любой ортогональный оператор в~$V$, перестановочный с~указанным действием. Далее, можно считать, что $V_0=0$, поскольку $V_0$ и~$V_0^{\perp}$\т $G$\д инвариантные подпространства, а~действие $G^0\cln V_0$ тождественно. Тогда размерности неприводимых компонент представления $G^0\cln V$ (с учётом кратностей) суть $n_1\sco n_l$.

Для доказательства теоремы~\ref{quat} потребуется следующее вспомогательное утверждение.

\begin{stm}\label{prod} Для любых чисел $k\in\N$, $d_1\sco d_k\in\R$ и~$d\in\R_{\ge0}$ существует единственное вещественное число~$t$, удовлетворяющее условиям $\fa i\;t+d_i\ge0$ и~$\prodl{i=1}{k}(t\bw+d_i)\bw=d$.
\end{stm}

\begin{proof} Без ограничения общности можно считать, что $d_1\sle d_k$.

Ясно, что $\fa i\;t+d_i\ge0\Lra t\ge-d_1$. Ограничение полиномиальной функции $f\cln\R\to\R,\,t\to\prodl{i=1}{k}(t+d_i)$ на промежуток $[-d_1;+\bes)$ строго возрастает, обращается в~нуль в~точке $-d_1$, стремится к~$+\bes$ при $t\to+\bes$ и~потому принимает каждое неотрицательное значение ровно в~одной точке. В~частности, $\exu t\in[-d_1;+\bes)\cln f(t)=d$.
\end{proof}

Дальнейшая часть доказательства теоремы~\ref{quat} будем проходить отдельно для каждого из случаев~\ref{case44}---\ref{case333}, указанных в~её формулировке.

\begin{cas}\label{cas44} $l=2$, $n_1=n_2=4$.
\end{cas}

Для каждой прямоугольной матрицы $A=(a_{ij})$ над телом~$\Hbb$ определена прямоугольная матрица $A^*:=\br{\ol{A}}^T$. Если число столбцов матрицы~$A_1$ равно числу строк матрицы~$A_2$, то $(A_1A_2)^*=A_2^*A_1^*$.

Пространство~$V$ можно отождествить с~пространством $\Mat_{2\times1}(\Hbb)$, в~котором всякий вектор~$z$ имеет скалярный квадрат $z^*z\in\R_{\ge0}$, а~кватернион~$h$ с~модулем~$1$, рассматриваемый как элемент группы $\SU_2$, действует умножением матриц $2\times1$ справа на матрицу $1\times1$ с~единственным элементом $h^{-1}$, то есть правым умножением координат векторов-столбцов на $h^{-1}$.

Выясним, как устроен фактор $V/G^0$. Если два вектора-столбца $z=\rbmat{z_1\\z_2}$ и~$w\bw=\rbmat{w_1\\w_2}$, $z_i,w_i\in\Hbb$, лежат в~одной орбите действия $G^0\cln V$, то $|z_i|=|w_i|$ ($i=1,2$) и~$z_1\ol{z_2}=w_1\ol{w_2}$. Обратное утверждение тоже верно. В~самом деле, если модули координат векторов $z$ и~$w$ попарно равны и~среди этих модулей есть нулевой, то $w\in G^0z$, а~если модули их координат попарно равны и~все отличны от нуля, то $w\in G^0z\Lra z_1\ol{z_2}=w_1\ol{w_2}$.

Отображение $\pi\cln V\to\R^2\oplus\Hbb,\,\rbmat{z_1\\z_2}\to\br{|z_1|^2;|z_2|^2;z_1\ol{z_2}}$ постоянно на орбитах действия $G^0\cln V$ и~разделяет их, а~его образом является подмножество
\equ{
M:=\bc{(d_1;d_2;\la)\cln d_i\in\R_{\ge0},\,\la\in\Hbb,\,d_1d_2=|\la|^2}\subs\R^2\oplus\Hbb.}
Далее, отображение $\ph\cln M\to\R\oplus\Hbb,\,(d_1;d_2;\la)\to(d_1-d_2;\la)$ биективно\: прообраз вектора $(d;\la)\in\R\oplus\Hbb$ совпадает с~подмножеством $\bc{(t+d;t;\la)\cln t\in\R_{\ge0},\,t+d\in\R_{\ge0},\,t(t+d)\bw=|\la|^2}\subs M$ и, согласно утверждению~\ref{prod}, состоит ровно из одного элемента. Значит, гладкое отображение $\pi'\cln V\to\R\oplus\Hbb,\,\rbmat{z_1\\z_2}\to\br{|z_1|^2-|z_2|^2;z_1\ol{z_2}}$, равное $\ph\circ\pi$, сюръективно, постоянно на орбитах группы~$G^0$ и~разделяет их.

Первое утверждение теоремы доказано. Докажем теперь второе.

Подпространство $S\subs\Mat_{2\times2}(\Hbb)$ всех эрмитовых кватернионных матриц содержит прямую $\R E$\~ диагональные элементы любой матрицы из~$S$ вещественны. Линейное отображение $S\to\R\oplus\Hbb,\,D\to(d_{11}-d_{22};d_{12})$  сюръективно, имеет ядро $\R E$ и~потому индуцирует линейный изоморфизм $S\hb{(\R E)}\to\R\oplus\Hbb$, при помощи которого можно отождествить фактор $V/G^0$ с~пространством $S\hb{(\R E)}$. Отображение
факторизации $\pi'\cln V\to S\hb{(\R E)}$ примет вид $\rbmat{z_1\\z_2}\to\rbmat{|z_1|^2&z_1\ol{z_2}\\z_2\ol{z_1}&|z_2|^2}+\R E$, то есть
$z\to zz^*+\R E$ (ясно, что $zz^*\in S$).

Ассоциативная алгебра эндоморфизмов представления $G^0\cln V$ канонически отождествляется с~ассоциативной алгеброй $\Mat_{2\times2}(\Hbb)$, действующей на пространстве $V\bw=\Mat_{2\times1}(\Hbb)$ умножениями слева. Предположим, что оператор $B\in\Mat_{2\times2}(\Hbb)$ в~пространстве~$V$ ортогонален. Тогда $BB^*=E$. Если $z\in\Mat_{2\times1}(\Hbb)$\т произвольный вектор, то $Bz(Bz)^*=B\br{zz^*}B^*$. Поскольку линейный оператор
$S\to S,\,D\to BDB^*=BDB^{-1}$ переводит в~себя матрицу~$E$ и~прямую $\R E$, оператор~$B$ действует линейно на факторе $V/G^0\cong S\hb{(\R E)}$.

\begin{cas}\label{cas43} $l=2$, $n_1=4$, $n_2=3$.
\end{cas}

Доказательство во многом повторяет доказательство случая~\ref{cas44}.

Пространство~$V$ можно отождествить с~пространством $\Hbb\oplus\Hbb_0$ (см. \S\,\ref{su2}), в~котором всякий вектор $(z;z_0)$ имеет скалярный квадрат $|z|^2+|z_0|^2$, а~кватернион~$h$ с~модулем~$1$, рассматриваемый как элемент группы $\SU_2$, действует по формуле $(z;z_0)\bw\to(zh^{-1};hz_0h^{-1})$.

Выясним, как устроен фактор $V/G^0$. Если два вектора $(z;z_0)$ и~$(w;w_0)$ из $\Hbb\bw\oplus\Hbb_0$ лежат в~одной орбите действия $G^0\cln V$, то $|z|=|w|$, $|z_0|=|w_0|$ и~$zz_0\ol{z}=ww_0\ol{w}$. Обратное утверждение тоже верно. В~самом деле, если модули координат векторов $(z;z_0)$ и~$(w;w_0)$ по прямым слагаемым $\Hbb$ и~$\Hbb_0$ попарно равны и~среди этих модулей есть нулевой, то $(w;w_0)\in G^0(z;z_0)$, а~если модули их координат попарно равны и~все отличны от нуля, то $(w;w_0)\in G^0(z;z_0)\Lra zz_0\ol{z}=ww_0\ol{w}$.

Отображение $\pi\cln V\to\R^2\oplus\Hbb_0,\,(z;z_0)\to\br{|z|^2;|z_0|^2;zz_0\ol{z}}$ постоянно на орбитах действия $G^0\cln V$ и~разделяет их, а~его образом является подмножество
\equ{
M:=\bc{(d;d_0;\la)\cln d,d_0\in\R_{\ge0},\,\la\in\Hbb_0,\,d^2d_0=|\la|^2}\subs\R^2\oplus\Hbb_0.}
Далее, отображение $\ph\cln M\to\R\oplus\Hbb_0,\,(d;d_0;\la)\to(d-d_0;\la)$ биективно\: прообраз вектора $(d;\la)\in\R\oplus\Hbb_0$ совпадает с~подмножеством $\bc{(t+d;t;\la)\cln t\in\R_{\ge0},\,t+d\in\R_{\ge0},\,t(t\bw+d)^2\bw=|\la|^2}\subs M$ и, согласно утверждению~\ref{prod}, состоит ровно из одного элемента. Значит, гладкое отображение $\pi'\cln V\to\R\oplus\Hbb_0,\,(z;z_0)\to\br{|z|^2-|z_0|^2;zz_0\ol{z}}$, равное $\ph\circ\pi$, сюръективно, постоянно на орбитах группы~$G^0$ и~разделяет их.

Рассмотрим произвольный эндоморфизм~$B$ представления $G^0\cln V$. Он переводит в~себя изотипные компоненты $\Hbb$ и~$\Hbb_0$ этого представления, действуя на первой из них левым сдвигом на кватернион~$b$, а~на второй\т скалярным оператором $b_0E$, $b_0\in\R$. Предположим, что $B\in\Or(V)$. Тогда $|b|=|b_0|=1$. Оператор~$B$ действует на факторе $V/G^0\cong\R\oplus\Hbb_0$ линейным преобразованием $(d;\la)\to(d;b_0b\la\ol{b})$\: если
$(z;z_0)\bw\in\Hbb\oplus\Hbb_0$\т некоторый вектор, то $B(z;z_0)=(bz;b_0z_0)$, причём $|bz|^2-|b_0z_0|^2=|z|^2-|z_0|^2$
и~$(bz)(b_0z_0)\ol{(bz)}=b_0\cdot b\br{zz_0\ol{z}}\ol{b}$. 

\begin{cas}\label{cas333} $l=3$, $n_1=n_2=n_3=3$.
\end{cas}

Пространство~$V$ можно отождествить с~пространством $\Mat_{3\times3}(\R)$, в~котором скалярное умножение имеет вид $(A_1,A_2)=\tr(A_1A_2^T)$, а~группа $\SO_3$ действует по формуле $C\cln A\to AC^{-1}$. Ассоциативная алгебра эндоморфизмов представления $G^0\cln V$ канонически отождествляется с
ассоциативной алгеброй $\Mat_{3\times3}(\R)$, действующей на пространстве $V=\Mat_{3\times3}(\R)$ умножениями слева. Оператор
$B\in\Mat_{3\times3}(\R)$ в~пространстве~$V$ ортогонален тогда и~только тогда, когда $BB^T=E$, $B\in\Or_3$.

Далее для доказательства теоремы~\ref{quat} понадобится следующая более общая конструкция.

Пусть $W$\т конечномерное евклидово (соотв. эрмитово) пространство над полем~$\F$, равным~$\R$ (соотв.~$\Cbb$), а~$A^*\in\End(W)$\т оператор, сопряжённый произвольному оператору $A\in\End(W)$. В~пространстве $\End(W)$ над полем~$\F$ все самосопряжённые операторы образуют вещественное подпространство $S\subs\End(W)$. В~этом пространстве рассмотрим подмножество $S_+$ всех неотрицательно определённых самосопряжённых операторов, а~в~пространстве $\End(W)\oplus\F$\т замкнутое подмножество $M:=\bc{(D;\la)\in S_+\times\F\cln\det D=|\la|^2}$.

В группе Ли $\GL(W)$ над полем~$\F$ все линейные операторы, сохраняющие скалярные произведения, образуют компактную вещественную подгруппу Ли $\Or(W):=\bc{C\bw\in\End(W)\cln CC^*=E}\bw\subs\GL(W)$. Компактная подгруппа Ли $\SO(W):=\Or(W)\cap\SL(W)$ действует линейно в~пространстве $\End(W)$ по формуле
\eqn{\label{repr}
\SO(W)\cln\End(W),\,C\cln A\to AC^{-1}.}
Для всякого $B\in\End(W)$ оператор $\End(W)\to\End(W),\,A\to BA$ перестановочен с~действием~\eqref{repr}.


Сформулируем теорему, являющуюся обобщением теоремы~\ref{quat} в~случае~\ref{case333}.

\begin{theorem}\label{repso} Фактор действия~\eqref{repr} диффеоморфен векторному пространству, причём на нём действует линейно любой оператор \eqn{\label{end}
\End(W)\to\End(W),\,A\to BA,}
где $B\in\Or(W)$.
\end{theorem}

Следующее утверждение является известным.

\begin{stm}\label{factor} Отображение $\pi\cln\End(W)\to S\oplus\F,\,A\to(AA^*;\det A)$ постоянно на орбитах действия~\eqref{repr} и~разделяет их, а~его образом является подмножество $M\subs S\oplus\F$.
\end{stm}

Очевидно, что вещественное пространство~$S$ содержит прямую $\R E$.

\begin{lemma}\label{biject} Отображение $\ph\cln M\to\br{S\hb{(\R E)}}\oplus\F,\,(D;\la)\to(D+\R E;\la)$ биективно.
\end{lemma}

\begin{proof} Прообраз вектора $(D+\R E;\la)\in\br{S\hb{(\R E)}}\oplus\F$, где $D\in S$, $\la\in\F$, под действием~$\ph$ совпадает с~подмножеством
$\bc{(tE+D;\la)\cln t\in\R,\,tE\bw+D\bw\in S_+,\,\det(tE\bw+D)\bw=|\la|^2}\subs M$.

Требуется доказать, что для любого оператора $D\in S$ и~числа $d\in\R_{\ge0}$ существует единственное вещественное число~$t$, удовлетворяющее условиям $tE+D\in S_+$ и~$\det(tE\bw+D)\bw=d$.

Оператор~$D$ в~некотором ортонормированном базисе пространства~$W$ записывается матрицей $\diag(d_1\sco d_k)$, где $k:=\dim_{\F}W$, $d_i\in\R$. Если $t\in\R$, то $tE+D\in S_+\bw\Lra\fa i\;t\bw+d_i\bw\ge0$, а~также $\det(tE+D)=\prodl{i=1}{k}(t+d_i)$. Осталось применить утверждение~\ref{prod}.
\end{proof}



Перейдём непосредственно к~доказательству теоремы~\ref{repso}.

Прежде всего, фактор действия~\eqref{repr} диффеоморфен векторному пространству $\br{S\hb{(\R E)}}\bw\oplus\F$, причём факторизацию можно задать формулой $\pi'\cln\End(W)\bw\to\br{S\hb{(\R E)}}\bw\oplus\F,\,A\to(AA^*+\R E;\det A)$. Чтобы это доказать, достаточно воспользоваться утверждением~\ref{factor} и~леммой~\ref{biject}, после чего положить $\pi':=\ph\circ\pi$.

Пусть теперь $B\in\Or(W)$. Если $A\in\End(W)$, то $BA(BA)^*=B\br{AA^*}B^*$, $\det(BA)\bw=(\det B)(\det A)$. Линейный оператор
$S\to S,\,D\to BDB^*=BDB^{-1}$ переводит в~себя матрицу~$E$ и~прямую $\R E$. Следовательно, оператор~\eqref{end} действует линейно на факторе действия~\eqref{repr}, диффеоморфном $\br{S\hb{(\R E)}}\oplus\F$.

Таким образом, нами доказана теорема~\ref{repso} и~вместе с~ней теорема~\ref{quat}.

\subsection{Доказательство теоремы~\ref{n54}}

Предположим, что группа~$G$ изоморфна одной из групп $\SU_2$ и~$\SO_3$, фактор её действия $G\cln V$ является гладким многообразием,
а~представление $\ggt\cln V_0^{\perp}$ есть прямая сумма неприводимых представлений $\ggt\cln V_1$ и~$\ggt\cln V_2$
размерностей $5$ и~$4$ соответственно. Поскольку все неприводимые представления группы $\SO_3$ имеют нечётную размерность,
группа~$G$ на самом деле изоморфна $\SU_2$. Согласно лемме~\ref{odd}, найдётся вектор $v\in V_1$ с~нетривиальным конечным стабилизатором.
В~силу леммы~\ref{refstab}, пересечение $G_v\cap\Om$ содержит некоторый нетривиальный элемент $g\in G$. Кроме того,
$\ggt v\subs V_1$, $V_2\subs N_v$, откуда $\dim\br{(E-g)V_2}\le\dim\br{(E-g)N_v}=\om(g)\le2$. С~другой стороны,
если отождествить группу $\SU_2$ с~группой $\bc{\la\in\Hbb\cln|\la|=1}$, то $g\ne1$, $\Hbb(1-g^{-1})=\Hbb$, $(E-g)V_2=V_2$, что противоречит неравенству $\dim\br{(E-g)V_2}\le2$.

Теорема~\ref{n54} доказана.

\subsection{Доказательство теоремы~\ref{n7}}

Можно считать, что подпространство~$V_0$ тривиально, поскольку группа $G=G^0$ действует на нём тождественно.

Итак, $G\cln V$\т неприводимое представление размерности $n_1=7$. Оно совпадает с~представлением $\SO_3\cln V_{\R}(7)$, а~его проективизация непрерывно отождествляется с~тавтологическим действием $\SO_3\cln SP^3(\RP^2)$. Исходя из этого, опишем фактор проективного действия
$\SO_3\cln PV_{\R}(7)$.

Топологическое пространство $SP^3(\RP^2)$ удобно понимать как фактор декартова куба единичной сферы евклидова пространства~$\R^3$ по действию конечной группы~$H$, порождённой всевозможными перестановками векторов и~сменами их знаков. Действие $H\cln(S^2)^3$ перестановочно с~тавтологическим действием $\Or_3\cln(S^2)^3$. Кроме того, $\Or_3=\SO_3\times\{\pm E\}$, а~оператор $(-E)$ действует на $(S^2)^3$ так же, как и~элемент группы~$H$, меняющий знаки всех трёх векторов. Следовательно, фактор пространства $SP^3(\RP^2)$ по действию $\SO_3$ гомеоморфен фактору индуцированного действия группы~$H$ на факторе $\hbr{(S^2)^3}\Or_3$. Отображение факторизации $(S^2)^3\thra\hbr{(S^2)^3}\Or_3$ задаётся формулой
\eqn{\label{facor}
\rh\cln(S^2)^3\to\R^3,\,(v_1;v_2;v_3)\to\br{(v_2,v_3);(v_3,v_1);(v_1,v_2)},}
причём
\eqn{\label{ddef}
D:=\rh\br{(S^2)^3}=\bc{x\in\R^3\cln A(x)\ge0}\subs\R^3,\;A(x):=\rbmat{1&x_3&x_2\\x_3&1&x_1\\x_2&x_1&1}.}
При перестановке векторов $v_1,v_2,v_3$ их попарные скалярные произведения также переставляются, а~при смене знака одного из этих векторов два попарных скалярных произведения меняют знак, а~третье остаётся прежним. Значит, индуцированное действие группы~$H$ на $\hbr{(S^2)^3}\Or_3\cong D$ совпадает с~ограничением действия $H'\cln\R^3$ линейной группы $H'$, порождённой всевозможными перестановками координат и~сменами их знаков в~чётном количестве, на инвариантное подмножество $D\subs\R^3$. Другими словами, $H'$\т конечная группа отражений, соответствующая системе корней $\De:=\{\pm\ep_i\pm\ep_j\cln i\ne j\}$, где $\{\ep_1;\ep_2;\ep_3\}$\т стандартный базис в~$\R^3$.

Подмножество $C:=\bc{x\in\R^3\cln x_1\ge x_2\ge|x_3|}\subs\R^3$ 
является камерой Вейля, отвечающей системе простых корней $\{\ep_1-\ep_2;\ep_2-\ep_3;\ep_2+\ep_3\}\subs\De$, и~потому пересекает каждую орбиту действия $H'\cln\R^3$ ровно в~одной точке. Следовательно, фактор $D/H'$ гомеоморфен $M:=C\cap D\subs\R^3$, откуда
$\hbr{\br{SP^3(\RP^2)}}\SO_3\cong\hbr{\br{\hbr{(S^2)^3}\Or_3}}H\cong M$.

Отображения факторизации $\rh\cln(S^2)^3\thra\hbr{(S^2)^3}\Or_3\cong D$, $\ta\cln D\thra D/H'\cong M$,
$\ta'\cln(S^2)^3\thra(S^2)^3/H\cong SP^3(\RP^2)$ и~$\rh'\cln SP^3(\RP^2)\thra\hbr{\br{SP^3(\RP^2)}}\SO_3\cong M$ удовлетворяют равенству
$\rh'\circ\ta'\equiv\ta\circ\rh$, то есть диаграмма
\eqn{\label{diagr}
\xymatrix{
(S^2)^3       \ar@{->>}[r]^-{\rh}  \ar@{->>}[d]^-{\ta'} & D \ar@{->>}[d]^-{\ta}\\
\lp{SP^3}(\RP^2) \ar@{->>}[r]^-{\rh'}                   & M}}
коммутативна.

Из определений подмножеств $C$ и~$D$ вытекает замкнутость и~выпуклость каждого из них, а~значит, и~их пересечения~$M$. Поскольку камера~$C$ соответствует системе положительных корней $\{\ep_i\pm\ep_j\cln i~j\}\subs\De$, подмножество~$M$ задаётся системой неравенств
\begin{align}
&&x_i+x_j&\ge0,&&x_i-x_j\ge0&(i>j);&&\ &&\label{weyl}\\
&&1+x_i&\ge0,&&1-x_i\ge0;\label{min}\\
&&\de(x)&\ge0,&&\de(x):=\det A(x).\label{det}
\end{align}

Докажем, что точка~$x$ множества~$M$ принадлежит его границе тогда и~только тогда, когда в~ней по крайней мере одно из неравенств \eqref{weyl}, \eqref{min} и~\eqref{det} обращается в~равенство. Достаточно доказать утверждение <<тогда>>, так как в~левых частях этих неравенств стоят непрерывные функции. Если $x\in\Int M$, то в~точке~$x$ все линейные неравенства \eqref{weyl} и~\eqref{min} являются строгими, в~частности, $1>x_1>x_2>\pm x_3$, $1\bw>x_1\bw>x_2\bw>|x_3|\bw\ge0$, $x_1>x_2|x_3|=|x_2x_3|\ge x_2x_3$,
$\frac{\pd\de}{\pd x_1}(x)=\frac{\pd}{\pd x_1}\br{2x_1x_2x_3+1-(x_1^2+x_2^2+x_3^2)}\bw=2(x_2x_3-x_1)<0$, $\nab_x\de\ne0$, вследствие чего неравенство~\eqref{det} оказывается строгим в~точке~$x$.

Допустим, что в~точке $x\in M$ одно из неравенств~\eqref{min} обращается в~равенство. Тогда в~этой точке обращается в~равенство по крайней мере одно из неравенств~\eqref{weyl}. Действительно, если $x\in M$, $x_i=\pm1$ и~$\{i;j;k\}=\{1;2;3\}$, то
$0\le\de(x)=2x_ix_jx_k\bw+1\bw-(x_i^2+x_j^2+x_k^2)=\pm2x_jx_k-x_j^2-x_k^2=-(x_j\mp x_k)^2$, $(x_j\mp x_k)^2\le0$, $x_j\mp x_k=0$.

Таким образом, фактор $\hbr{\br{SP^3(\RP^2)}}\SO_3$ гомеоморфен выпуклому компакту $M\bw=C\cap D\subs\R^3$, задаваемому неравенствами \eqref{weyl}, \eqref{min} и~\eqref{det}, причём $\pd M$ есть подмножество всех векторов из~$M$, в~которых по крайней мере одно из неравенств \eqref{weyl} и~\eqref{det} обращается в~равенство\:
\eqn{\label{dM}
\pd M=\bc{x\in M\cln(x_1^2-x_2^2)(x_2^2-x_3^2)(x_3^2-x_1^2)\det A(x)=0}.}

Вернёмся к~исходному представлению $G\cln V$.

Для всякого подмножества $U\subs V$, замкнутого относительно умножения на числа из $\R^*$, будем обозначать через $PU\subs PV$ множество прямых пространства~$V$, содержащихся в~$U\cup\{0\}$.

Положим $U:=\{v\in V\cln-v\notin Gv\}\subs V$, $F:=V\sm U=\{v\in V\cln-v\in Gv\}\subs V$.

Очевидно, что $GU=\R^*U=U$ и~$GF=\R^*F=F$. Группа~$G$ при действии на единичной сфере $S\subs V$ переводит в~себя подмножества $S\cap U$ и~$S\cap F$, а~при проективном действии на $PV$\т подмножества $PU$ и~$PF$, причём $S=(S\cap U)\sqcup(S\cap F)$, $S/G\bw=(S\cap U)/G\sqcup(S\cap F)/G$,
$PV=PU\sqcup PF$, $PV/G=PU/G\sqcup PF/G$. Двулистное накрытие $S\thra PV,\,v\to\R v$ перестановочно с~действием группы~$G$.
Индуцированное отображение $S/G\thra PV/G$ при ограничении на $(S\cap F)/G$ даёт гомеоморфизм $(S\cap F)/G\to PF/G$, а~при ограничении
на $(S\cap U)/G$\т двулистное накрытие $(S\cap U)/G\thra PU/G$.

В группе $\SU_2$ каждый класс сопряжённости имеет непустое пересечение с~подгруппой $\T:=\{\diag(\la;\ol{\la})\cln\la\in\Cbb,\,|\la|=1\}\subs\SU_2$. Поэтому $F=GF'$, где $F':=\{v\in V\cln-v\bw\in\T v\}\bw\subs V$. Кроме того, $\R^*F'=F'$, $PF=G(PF')$.

Многочлен $f=\sums{|k|\le3}c_kx^{3+k}y^{3-k}\in V_{\R}(7)$ принадлежит подмножеству $F'$ тогда и~только тогда, когда найдётся число $\la\in\Cbb$ с~модулем~$1$, такое что $(\la^{2k}+1)c_k=0$ для всех~$k$, $|k|\bw\le3$. Последнее, в~свою очередь, эквивалентно тому, что все числа~$k$, удовлетворяющие условиям $|k|\le3$ и~$c_k\ne0$, отличны от нуля и~включают в~своё разложение на простые множители число~$2$ с~одной и~той же кратностью, то есть что многочлен~$f$ принадлежит объединению комплексных подпространств $W_1:=\ha{x^{3+k}y^{3-k}\cln|k|\in\{1;3\}}_{\Cbb}$ и~$W_2:=\ha{x^{3+k}y^{3-k}\cln|k|=2}_{\Cbb}$ пространства $V(7)$.

Итак, $F'=(W_1\cup W_2)\cap V_{\R}(7)=F'_1\cup F'_2$, где $F'_i:=W_i\cap V_{\R}(7)$. Далее, $PF'=PF'_1\cup PF'_2$,
\eqn{\label{PF}
PF=G(PF')=G(PF'_1)\cup G(PF'_2).}

Теперь опишем $PF$, $PF'$, $PF'_i$ и~$G(PF'_i)$ ($i=1,2$) как подмножества топологического пространства $SP^3(\RP^2)$ с~учётом гомеоморфизма $\SO_3$\д пространств $PV$ и~$SP^3(\RP^2)$. Очевидно, что равенство~\eqref{PF} переносится с~пространства $PV$ на $SP^3(\RP^2)$.

Многочлен $f\in V_{\R}(7)$ принадлежит $W_1=\ha{x^6;x^4y^2;x^2y^4;y^6}_{\Cbb}$ тогда и~только тогда, когда его разложение на линейные множители вместе с~любым множителем $ax+by$ включает в~себя и~$ax-by$\т что равносильно, соответствующая неупорядоченная тройка прямых в~$\R^3$ переходит в~себя при повороте на угол~$\pi$ вокруг вертикальной прямой. Следовательно, $PF'_1\subs SP^3(\RP^2)$\т подмножество всех неупорядоченных троек прямых, переходящих в~себя при повороте на угол~$\pi$ вокруг вертикальной прямой, а~$G(PF'_1)\subs SP^3(\RP^2)$\т подмножество всех неупорядоченных троек прямых, переходящих в~себя при повороте на угол~$\pi$ вокруг некоторой прямой.

Многочлен $f=\sums{|k|\le3}c_kx^{3+k}y^{3-k}\in V_{\R}(7)$ принадлежит $W_2=\ha{x^5y;xy^5}_{\Cbb}$ тогда и~только тогда, когда
\eqn{\label{lin}
f=c_2x^5y+c_{-2}xy^5=xy(c_2x^4+c_{-2}y^4)=xy(c_2x^4+\ol{c_2}y^4)
=xy(ax+by)(ax+iby)(ax-by)(ax-iby),}
причём $a^4=c_2$, $b^4=-\ol{c_2}$, $|a|=|b|$. Пользуясь разложением~\eqref{lin} многочлена~$f$ на линейные множители, получаем, что $PF'_2\subs SP^3(\RP^2)$\т подмножество всех неупорядоченных троек прямых, одна из которых вертикальна, а~две другие горизонтальны и~ортогональны друг другу, то есть всех неупорядоченных троек попарно ортогональных прямых, одна из которых вертикальна. В~свою очередь, $G(PF'_2)\subs SP^3(\RP^2)$\т подмножество всех неупорядоченных троек попарно ортогональных прямых.

Согласно~\eqref{PF}, подмножество $PF\subs SP^3(\RP^2)$ состоит в~точности из всех неупорядоченных троек прямых, переходящих в~себя при повороте на угол~$\pi$ вокруг некоторой прямой либо включающих в~себя три попарно ортогональные прямые.

Докажем, что тройка прямых в~$\R^3$ принадлежит $PF$ тогда и~только тогда, когда либо все три прямые лежат в~одной плоскости, либо из трёх попарных углов между ними найдутся два равных.

Если все три прямые лежат в~одной плоскости, то при повороте на угол~$\pi$ вокруг прямой, ортогональной этой плоскости, каждая из них переходит в~себя. Если два из трёх попарных углов между ними равны, то можно выбрать на каждой прямой по вектору единичной длины, так чтобы выбранные векторы $v_1,v_2,v_3$ удовлетворяли равенству $(v_1,v_3)=(v_2,v_3)$. Тогда подпространство всех векторов, ортогональных векторам $v_1+v_2$ и~$v_3$, нетривиально и~содержит вектор $v_1-v_2$, а~значит, и~некоторую прямую, проходящую через $v_1-v_2$. Пусть $R$\т поворот на угол~$\pi$ вокруг этой прямой. Ясно, что $R(v_1-v_2)=v_1-v_2$, $R(v_1+v_2)=-(v_1+v_2)$, $Rv_3=-v_3$, $Rv_1=-v_2$, $Rv_2=-v_1$, $R(\R v_1)=\R v_2$, $R(\R v_2)=\R v_1$,
$R(\R v_3)=\R v_3$.

Обратно, предположим, что все три попарных угла между прямыми $\R v_1,\R v_2,\R v_3$, где $v_i\in\R^3\sm\{0\}$, различны. Тогда одна из прямых
$\R v_i$ неортогональна двум другим. Можно считать, что
\eqn{\label{noort}
(v_1,v_3)\ne0;\quad(v_2,v_3)\ne0.}
В частности, указанные три прямые не могут быть попарно ортогональными. Допустим, что ортогональный нескалярный оператор~$R$ переводит в~себя их неупорядоченную тройку. Он сохраняет углы между прямыми, откуда $R(\R v_i)=\R v_i$, $R v_i=\la_iv_i$, $\la_i=\pm1$ ($i=1,2,3$). Согласно~\eqref{noort}, $\la_1\la_3=\la_2\la_3=1$, $\la_1=\la_2=\la_3$, то есть ограничение оператора~$R$ на линейную оболочку векторов $v_1,v_2,v_3$ есть скалярный оператор. Следовательно, прямые $\R v_i$ ($i=1,2,3$) лежат в~одной плоскости.

Итак, $PF\subs SP^3(\RP^2)$\т подмножество всех неупорядоченных троек прямых, лежащих в~одной плоскости либо образующих три попарных угла, среди которых найдутся два равных. Теперь, пользуясь коммутативной диаграммой~\eqref{diagr}, найдём образ этого подмножества при отображении факторизации
$\rh'\cln SP^3(\RP^2)\thra M$. Легко видеть, что $(\ta')^{-1}(PF)\subs(S^2)^3$ есть подмножество всех упорядоченных троек векторов
$v_1,v_2,v_3\in S^2$, линейно зависимых либо удовлетворяющих равенству
$\br{(v_1,v_2)^2-(v_2,v_3)^2}\br{(v_2,v_3)^2-(v_3,v_1)^2}\br{(v_3,v_1)^2-(v_1,v_2)^2}=0$. Далее,
\equ{
\rh\br{(\ta')^{-1}(PF)}=\bc{x\in D\cln(x_1^2-x_2^2)(x_2^2-x_3^2)(x_3^2-x_1^2)\det A(x)=0}\subs D.}
Учитывая равенство $M=C\cap D$ и~формулу~\eqref{dM}, получаем, что $\rh\br{(\ta')^{-1}(PF)}\cap C\bw=\pd M$. Образ $H'$\д инвариантного подмножества $\rh\br{(\ta')^{-1}(PF)}$ при отображении $\ta\cln D\bw\thra D/H'\cong M$ совпадает с~$\rh\br{(\ta')^{-1}(PF)}\cap C=\pd M$. Поэтому
$\rh'(PF)\bw=\rh'\bw\circ\ta'\br{(\ta')^{-1}(PF)}=\ta\circ\rh\br{(\ta')^{-1}(PF)}=\pd M$.

Из вышеизложенного можно заключить, что существует гомеоморфизм $PV/G\bw\to M$, переводящий подмножества $PF/G$ и~$PU/G$ соответственно в~$\pd M$ и
$\Int M$. Подмножество $\Int M\subs\R^3$ выпукло как внутренность выпуклого подмножества $M\subs\R^3$, а~значит, односвязно. Гомеоморфное ему подмножество $PU/G\subs PV/G$ также односвязно. Напомним, что отображение $S/G\thra PV/G$, $Gv\to G(\R v)$ при ограничении на
$(S\cap F)/G$ даёт гомеоморфизм $(S\cap F)/G\to PF/G$, а~при ограничении на $(S\cap U)/G$\т двулистное накрытие $(S\cap U)/G\thra PU/G$. Данное накрытие тривиально, поскольку его база $PU/G$ односвязна. Отсюда фактор $S/G$ гомеоморфен топологическому пространству
$\br{M\times\{\pm1\}}\hb{\bc{(x;1)\sim(x;-1)\cln x\in\pd M}}$. Кроме того, имеется гомеоморфизм из трёхмерного замкнутого шара в~подмножество $M\subs\R^3$, переводящий граничную сферу в~$\pd M$. Следовательно, фактор сферы~$S$ по действию~$G$ гомеоморфен несвязному объединению двух экземпляров трёхмерного замкнутого шара с~отождествлением двух экземпляров каждой граничной точки. Таким образом, $S/G\cong S^3$, и~$V/G\cong\R^4$.

Теорема~\ref{n7} доказана.

\subsection{Доказательство теоремы~\ref{n8}}

Аналогично доказательству предыдущей теоремы, можно считать, что $V_0=0$, и~$G\cln V$\т неприводимое представление размерности $n_1=8$.

Поскольку все неприводимые представления группы $\SO_3$ имеют нечётную размерность, группа~$G$ изоморфна~$\SU_2$, а~представление $G\cln V$ совпадает с~представлением $\SU_2\cln V(4)$ (см. \S\,\ref{su2}).

Достаточно доказать, что фактор единичной сферы $S\subs V$ по действию~$G$ гомеоморфен~$S^4$.

Будем обозначать через~$S^3$ единичную сферу эрмитова пространства $\Cbb^2\bw=\Mat_{2\times1}(\Cbb)$ со скалярным произведением $(v_1,v_2):=v_1^*v_2$. Группа $G\cong\SU_2$ действует на $\Cbb^2\sm\{0\}$ и~$S^3$ как на инвариантных подмножествах для её тавтологического представления в~пространстве~$\Cbb^2$.

Непрерывное отображение $\pi\cln\br{\Cbb^2\sm\{0\}}^3\to V\sm\{0\}$, определённое по формуле $\br{\pi(v_1;v_2;v_3)}(v):=\prodl{j=1}{3}(v_j,v)$ ($v\in\Cbb^2$), перестановочно с~действием~$G$. Данное отображение сюръективно, а~его слои суть орбиты действия на $\br{\Cbb^2\sm\{0\}}^3$ группы~$H_0$ всех преобразований $(v_1;v_2;v_3)\to(\la_1v_{\si(1)};\la_2v_{\si(2)};\la_3v_{\si(3)})$, где $\si\in S_3$, $\la_j\in\Cbb$, $\la_1\la_2\la_3=1$: каждый многочлен из $V(4)\sm\{0\}$ разлагается на линейные множители единственным образом с~точностью до их перестановки и~умножения на комплексные числа с~произведением~$1$.

Непрерывная функция $r\cln\br{\Cbb^2\sm\{0\}}^3\to\R_{>0},\,(v_1;v_2;v_3)\to\prodl{j=1}{3}\hn{v_j}$ постоянна на орбитах действий групп $H_0$ и~$G$ на $\br{\Cbb^2\sm\{0\}}^3$. Поэтому существует непрерывная функция $r'\cln V\sm\{0\}\bw\to\R_{>0}$, постоянная на орбитах действия $G\cln\br{V\sm\{0\}}$, такая что $r'\circ\pi\equiv r$. Имеем $r(\la_1v_1;\la_2v_2;\la_3v_3)=\Br{\prodl{j=1}{3}|\la_j|}r(v_1;v_2;v_3)$ и~$\pi(\la_1v_1;\la_2v_2;\la_3v_3)\bw=\Br{\prodl{j=1}{3}\ol{\la_j}}\pi(v_1;v_2;v_3)$, где $\la_j\in\Cbb^*$ и~$v_j\in\Cbb^2\sm\{0\}$, откуда
$r'(\la f)=|\la|r'(f)$ для любых $\la\in\Cbb^*$ и~$f\in V\sm\{0\}$. Значит, $G$\д инвариантное подмножество $S':=(r')^{-1}(1)\subs V\sm\{0\}$ пересекает каждое подмножество $(\R_{>0})v$, $v\in V\sm\{0\}$, ровно в~одной точке, а~расслоение $V\sm\{0\}\thra S,\,v\bw\to\frac{v}{\hn{v}}$ при ограничении на $S'$ даёт гомеоморфизм $G$\д пространств $S'$ и~$S$.

Следовательно, $S/G\cong S'/G$, и~задача свелась к~доказательству соотношения $S'/G\cong S^4$.

Заметим, что $\pi^{-1}(S')=r^{-1}(1)=H_0\br{(S^3)^3}$, $S'=\pi\Br{H_0\br{(S^3)^3}}=\pi\br{(S^3)^3}$. Слои отображения~$\pi|_{(S^3)^3}$ суть пересечения орбит действия $H_0\cln\br{\Cbb^2\sm\{0\}}^3$ с~подмножеством $(S^3)^3$, то есть орбиты действия на $(S^3)^3$ группы всех преобразований $(v_1;v_2;v_3)\bw\to(\la_1v_{\si(1)};\la_2v_{\si(2)};\la_3v_{\si(3)})$, где $\si\in S_3$, $\la_j\in\T$, $\la_1\la_2\la_3=1$. Что же касается композиции $(S^3)^3\thra S'\thra S'/G$ отображения $\pi|_{(S^3)^3}\cln(S^3)^3\thra S'$ и~отображения факторизации $S'\thra S'/G$, то её слои суть орбиты действия на $(S^3)^3$ группы~$H$ всех преобразований
\eqn{\label{elh}
(v_1;v_2;v_3)\to(\la_1gv_{\si(1)};\la_2gv_{\si(2)};\la_3gv_{\si(3)}),\quad
\quad\si\in S_3,\;g\in\SU_2,\;\la_j\in\T,\;\la_1\la_2\la_3=1,}
откуда $S'/G\cong(S^3)^3/H$. Теперь осталось доказать, что $(S^3)^3/H\cong S^4$.

Канонический гомоморфизм $\SU_2\thra\SO_3$ и~тавтологическое действие $\SO_3\cln S^2$ определяют действие $H\cln(S^2)^3$, при котором элемент~\eqref{elh} переводит точку $(v_1;v_2;v_3)$, где $v_j\in S^2\subs\R^3$, в~точку $(gv_{\si(1)};gv_{\si(2)};gv_{\si(3)})$. Композиция отображения $S^3\bw\thra\CP^1,\,v\to\Cbb v$ и~стереографической проекции $\CP^1\to S^2$ (известная как расслоение Хопфа $S^3\thra S^2$) индуцирует отображение $\ta'\cln(S^3)^3\thra(S^2)^3$, перестановочное с~действием~$H$.

Фактор действия $H\cln(S^2)^3$ гомеоморфен фактору тавтологического действия $\SO_3\cln SP^3(S^2)$. Напомним, что фактор тавтологического действия $\Or_3\cln(S^2)^3$ гомеоморфен выпуклому компактному подмножеству $D\subs\R^3$, определённому в~\eqref{ddef}, причём отображение факторизации $\rh\cln(S^2)^3\thra\hbr{(S^2)^3}\Or_3\cong D$ задаётся формулой \eqref{facor}. Фактор тавтологического действия $\Or_3\cln SP^3(S^2)$ гомеоморфен фактору ограничения действия линейной группы всевозможных перестановок координат в~$\R^3$ на инвариантное подмножество $D\subs\R^3$, то есть выпуклому компакту $M:=\bc{x\in D\cln x_1\ge x_2\ge x_3}\subs\R^3$.

Орбита действия $\Or_3\cln SP^3(S^2)$, содержащая неупорядоченную тройку векторов
$v_1,v_2,v_3\in S^2\subs\R^3$, есть объединение орбит неупорядоченных троек $(v_1;v_2;v_3)$ и~$(-v_1;-v_2;-v_3)$ для действия $\SO_3\cln SP^3(S^2)$. Эти орбиты совпадают тогда и~только тогда, когда либо векторы $v_1,v_2,v_3$ лежат в~одной плоскости, либо из трёх попарных углов между ними найдутся два равных. Следовательно, прообраз точки $x\in M\subs\R^3$ при естественном отображении
$\hbr{\br{SP^3(S^2)}}\SO_3\thra\hbr{\br{SP^3(S^2)}}\Or_3\cong M$ состоит из одной точки, если $(x_1-x_2)(x_2-x_3)(x_3-x_1)\det A(x)=0$ (что равносильно, $x\in\pd M$), и~из двух точек в~противном случае. Значит, указанное отображение при ограничении на прообраз $\pd M$ даёт гомеоморфизм, а~при ограничении на прообраз $\Int M$\т двулистное накрытие. Данное накрытие тривиально, поскольку его база $\Int M$ односвязна. Отсюда
\equ{
\hbr{\br{SP^3(S^2)}}\SO_3\cong\wh{M}:=\br{M\times\{\pm1\}}\hb{\bc{(x;1)\sim(x;-1)\cln x\in\pd M}},}
и
\eqn{\label{homeo}
(S^2)^3/H\cong\hbr{\br{SP^3(S^2)}}\SO_3\cong\wh{M}.}

Можно рассматривать $\pd M$ как подмножество пространства $\wh{M}$\: в~последнем два экземпляра границы $\pd M$ склеены в~один.

Пусть $\wh{\rh}\cln(S^2)^3\thra(S^2)^3/H\cong\wh{M}$\т отображение факторизации. Отображение $\ta\cln(S^3)^3\thra\wh{M}$, равное $\wh{\rh}\circ\ta'$, постоянно на орбитах действия $H\cln(S^3)^3$, и, более того,
\eqn{\label{fibh}
\fa p\in(S^3)^3\quad\quad\ta^{-1}\br{\ta(p)}=(\ta')^{-1}\br{H\ta'(p)}=(\ta')^{-1}\br{\ta'(Hp)},}
так как $\ta'$ есть отображение $H$\д пространств $(S^3)^3$ и~$(S^2)^3$. Кроме того,
\begin{gather}
\label{fib}
\fa p=(v_1;v_2;v_3)\in(S^3)^3\quad(\ta')^{-1}\br{\ta'(p)}=\bc{(\la_1v_1;\la_2v_2;\la_3v_3)\cln\la_j\in\T},\\
\label{fibta}
\fa p=(v_1;v_2;v_3)\in(S^3)^3\quad\ta^{-1}\br{\ta(p)}=
\bc{(\la_1gv_{\si(1)};\la_2gv_{\si(2)};\la_3gv_{\si(3)})\cln\si\in S_3,\;g\in\SU_2,\;\la_j\in\T}.
\end{gather}

Векторам $v_1,v_2\in\Cbb^2=\Mat_{2\times1}(\Cbb)$ сопоставим матрицу $\br{v_1\big|v_2}\in\Mat_{2\times2}(\Cbb)$ и~число $|v_1,v_2|:=\det\br{v_1\big|v_2}\in\Cbb$.

Для непрерывных функций $\ta'_j,\ta_j\cln(S^3)^3\to\Cbb$ ($j=1,2$), определённых формулами
\begin{align*}
\ta'_1(v_1;v_2;v_3)&:=(v_3,v_1)|v_3,v_2|+(v_3,v_2)|v_3,v_1|;\\
\ta'_2(v_1;v_2;v_3)&:=|v_1,v_2|^2;\\
 \ta_j(v_1;v_2;v_3)&:=\ta'_j(v_1;v_2;v_3)\cdot\ta'_j(v_2;v_3;v_1)\cdot\ta'_j(v_3;v_1;v_2)\quad(j=1,2),
\end{align*}
имеем $\ta'_j(v_1;v_2;v_3)=\ta'_j(v_2;v_1;v_3)$ ($v_1,v_2,v_3\in S^3$, $j=1,2$),
\eqn{\label{tha}
\fa j=1,2\,\fa p=(v_1;v_2;v_3)\in(S^3)^3\,\fa\la_1,\la_2,\la_3\in\T\quad\quad\ta_j(\la_1v_1;\la_2v_2;\la_3v_3)=(\la_1\la_2\la_3)^{2j}\ta_j(p),}
и~поэтому функции $\ta_1$ и~$\ta_2$ постоянны на орбитах действия $H\cln(S^3)^3$.

Выясним, как устроены подмножества $\ta_1^{-1}(0),\ta_2^{-1}(0)\subs(S^3)^3$.

Пусть $I'_1,I\"_1,I_2\subs\pd M\subs\wh{M}$\т образы непрерывных инъективных отображений
\equ{
\begin{aligned}
& \ga'_1\cln&\bs{-1;-\frac{1}{2}}&\to\pd M,&t&\to(2t^2-1;t;t),&&\ &&\\
&\ga\"_1\cln&\bs{-\frac{1}{2};1} &\to\pd M,&t&\to(t;t;2t^2-1),&&\ &&\\
&  \ga_2\cln&\bs{-1;1}&\to\pd M,&t&\to(1;t;t)&&\ &&
\end{aligned}}
соответственно, и~$I_1:=I'_1\cup I\"_1\subs\pd M\subs\wh{M}$.

\begin{prop} Справедливы равенства $\ta_j^{-1}(0)=\ta^{-1}(I_j)$ \ter{$j=1,2$}.
\end{prop}

\begin{proof} Очевидно, что $\ta_2^{-1}(0)\subs(S^3)^3$ есть подмножество всех троек векторов из $S^3\subs\Cbb^2$, среди которых найдутся два пропорциональных, то есть прообраз под действием $\ta'$ подмножества всех троек векторов двумерной сферы~$S^2$, по крайней мере два из которых равны. Указанное подмножество в~$(S^2)^3$ совпадает с~$(\wh{\rh})^{-1}(I_2)$.

Далее, для точки $p=(v_1;v_2;v_3)\in(S^3)^3$ равенство $\ta'_1(p)=0$ эквивалентно тому, что $\ta'(p)\in(S^2)^3$ есть тройка векторов, первые два из которых симметричны относительно прямой, натянутой на третий. Доказательство этого факта сводится к~случаю $v_3=\rbmat{1\\0}$ при помощи оператора из $\SU_2$, переводящего вектор~$v_3$ в~вектор $\rbmat{1\\0}$ (такой оператор существует). Если же $v_j=\rbmat{a_j\\b_j}$ ($j=1,2,3$), $a_3=1$, $b_3=0$, то $\ta'_1(p)=a_1b_2+a_2b_1$, что влечёт требуемое.

Значит, подмножество $\ta_1^{-1}(0)\subs(S^3)^3$ совпадает с~прообразом под действием $\ta'$ подмножества всех троек векторов сферы~$S^2$, какие-либо два из которых симметричны относительно прямой, натянутой на третий. Данное подмножество в~$(S^2)^3$ и~есть $(\wh{\rh})^{-1}(I_1)$.

Осталось воспользоваться соотношением $\ta=\wh{\rh}\circ\ta'$ и~получить нужные равенства. (Всюду при доказательстве использовался гомеоморфизм~\eqref{homeo}.)
\end{proof}

Из определений подмножеств $I'_1,I\"_1,I_2\subs\pd M$ вытекает, что указанные подмножества гомеоморфны отрезку, причём
$I'_1\cap I\"_1=\BC{\br{-\frac{1}{2};-\frac{1}{2};-\frac{1}{2}}}$, $I'_1\cap I_2=\bc{(1;-1;-1)}$ и~$I\"_1\cap I_2=\bc{(1;1;1)}$. Следовательно, $I_1\cong I_2\cong[0;1]$ и~$I_1\cap I_2=\bc{(1;1;1);(1;-1;-1)}$. Существует гомеоморфизм из подмножества $M\subs\R^3$ в~единичный замкнутый шар пространства~$\R^3$, переводящий $\pd M$ в~единичную сферу~$S^2$, а~подмножества $I_{1,2}\subs\pd M$\т в~полуокружности соответственно
$\bc{y\in\R^3\cln y_2=\pm\sqrt{1-y_1^2},\,y_3=0}\subs S^2$. В~свою очередь, существует гомеоморфизм из пространства $\wh{M}$ в~сферу
$S^3=\{y\in\R^4\cln(y,y)=1\}$, переводящий подмножество $\pd M$ в~двумерную сферу $\{y\in S^3\cln y_4=0\}$, а~подмножества $I_{1,2}$\т в~полуокружности соответственно $\bc{y\in\R^4\cln y_2=\pm\sqrt{1-y_1^2},\,y_3=y_4=0}$. Этот гомеоморфизм позволяет отождествить пространство $\wh{M}$ со сферой~$S^3$, подмножества $I_1$ и~$I_2$\т с~полуокружностями в~$S^3$, задаваемыми уравнениями $y_2=\sqrt{1-y_1^2}$ и~$y_2=-\sqrt{1-y_1^2}$ соответственно, а~отображение~$\ta$ рассматривать как отображение $\ta\cln(S^3)^3\thra S^3$, удовлетворяющее условиям $\ta_j^{-1}(0)=\ta^{-1}(I_j)$ ($j=1,2$).

Из формул \eqref{fibta} и~\eqref{tha} можно заключить, что\:
\begin{nums}{9}
\item отображения $(S^3)^3\to\R_{\ge0},\,p\to\bm{\ta_j(p)}$ ($j=1,2$) и~$(S^3)^3\to\Cbb,\,p\to\br{\ta_1(p)}^2\,\ol{\br{\ta_2(p)}}$ постоянны на слоях
$\ta$, а~значит, имеют вид соответственно $\al_1\circ\ta$, $\al_2\circ\ta$, $\al\circ\ta$ для непрерывных отображений
$\al_1,\al_2\cln S^3\to\R_{\ge0}$ и~$\al\cln S^3\to\Cbb$\~
\item если $K\subs S^3\times\Cbb^2$\т образ непрерывного отображения $\wt{\ta}\cln(S^3)^3\to S^3\times\Cbb^2,\,p\bw\to\br{\ta(p);\ta_1(p);\ta_2(p)}$,
то
\eqn{\label{orb}
\fa p\in(S^3)^3\quad\quad K\cap\Br{\bc{\ta(p)}\times\Cbb^2}=\BC{\br{\ta(p);\la\ta_1(p);\la^2\ta_2(p)}\cln\la\in\T}.}
\end{nums}

\begin{lemma} Фактор $(S^3)^3/H$ гомеоморфен пространству~$K$.
\end{lemma}

\begin{proof} Отображение $\wt{\ta}\cln(S^3)^3\thra K$ постоянно на орбитах действия $H\cln(S^3)^3$. Осталось доказать, что оно их разделяет.

Предположим, что $p_0,p\in(S^3)^3$ и~$\wt{\ta}(p_0)=\wt{\ta}(p)$, то есть $\ta(p_0)=\ta(p)$, $\ta_1(p_0)=\ta_1(p)$, $\ta_2(p_0)=\ta_2(p)$. Нужно доказать, что $p\in Hp_0$.

Будем считать, что $p_0=(v_1;v_2;v_3)$ и~$p=(\la_1v_1;\la_2v_2;\la_3v_3)$, где $\la_j\in\T$\: согласно \eqref{fibh} и~\eqref{fib}, к~этому можно свести заменой точек $p_0$ и~$p$ точками соответственно $hp_0$ и~$p$ для некоторого $h\in H$.

Достаточно доказать, что многочлены $\pi(p_0),\pi(p)\in S'$ принадлежат одной орбите действия $G\cln S'$, так как
$p\in Hp_0\Lra\pi(p)\in G\br{\pi(p_0)}$.

Ясно, что $\pi(p)=\ol{\la_0}\pi(p_0)$, где $\la_0:=\la_1\la_2\la_3\in\T$. В~силу~\eqref{tha}, $\ta_1(p_0)=\ta_1(p)\bw=\la_0^2\ta_1(p_0)$, то есть либо $\la_0^2=1$, либо $\ta_1(p_0)=0$.

Если $\la_0=\pm1$, то многочлен $\pi(p_0)$ переходит в~$\pi(p)$ под действием элемента $\pm E\bw\in\SU_2$.

Допустим, что $\ta_1(p_0)=0$.

Поскольку группа~$H$ включает в~себя всевозможные перестановки трёх векторов сферы~$S^3$, а~тавтологическое действие $\SU_2\cln S^3$ транзитивно, можно считать, что $\ta'_1(v_1;v_2;v_3)=0$, $v_j=\rbmat{a_j\\b_j}$ ($j=1,2,3$), $a_3=1$, $b_3=0$.

Имеем $a_1b_2+a_2b_1=\ta'_1(p_0)=0$. Значит,
\equ{
\ta_2(p_0)=b_1^2b_2^2(a_1b_2-a_2b_1)^2\bw=-4b_1^2b_2^2(a_1b_2)(a_2b_1)\bw=-4(a_1a_2)(b_1b_2)^3,}
а~коэффициенты многочлена $\pi(p_0)$, расположенные в~порядке убывания степени вхождения в~моном первой координаты, суть $\ol{(a_1a_2)},0,\ol{(b_1b_2)},0$. Согласно~\eqref{tha}, $\ta_2(p_0)\bw=\ta_2(p)=\la_0^4\ta_2(p_0)$,
$(a_1a_2)(b_1b_2)^3\bw=\la_0^4(a_1a_2)(b_1b_2)^3$. Таким образом,
\equ{
|a_1a_2|=|\la_0a_1a_2|,\quad|b_1b_2|=|\la_0b_1b_2|,\quad(a_1a_2)(b_1b_2)^3=(\la_0a_1a_2)(\la_0b_1b_2)^3.}
Следовательно, найдётся число $\la\in\T$, для которого $(\la_0a_1a_2)=\la^3(a_1a_2)$ и~$(\la_0b_1b_2)\bw=\la^{-1}(b_1b_2)$. Заметим, что элемент $g:=\diag(\la;\ol{\la})\in\SU_2$ переводит многочлен $\pi(p_0)$ в~многочлен с~коэффициентами $\ol{(\la^3a_1a_2)},0,\ol{(\la^{-1}b_1b_2)},0$, то есть $g\pi(p_0)=\ol{\la_0}\pi(p_0)\bw=\pi(p)$.
\end{proof}

Напомним, что требовалось доказать соотношение $(S^3)^3/H\cong S^4$. Теперь задача свелась к~доказательству того, что $K\cong S^4$.

Ясно, что $I_j=\al_j^{-1}(0)$ ($j=1,2$). Рассмотрим подмножества $U_1:=S^3\sm I_1$, $U_2:=S^3\sm I_2$ и~$U:=U_1\cap U_2=\al^{-1}(\Cbb^*)$ сферы~$S^3$. Подмножество $U=S^3\sm(I_1\cup I_2)=\{y\bw\in S^3\cln y_1^2+y_2^2<1\}$ гомеоморфно декартову произведению открытого круга на окружность. Отображение
$\al|_U\cln U\to\Cbb^*$ индуцирует гомоморфизм $\al_*\cln\pi_1(U)\to\pi_1(\Cbb^*)$, причём $\pi_1(U)\cong\pi_1(\Cbb^*)\cong\Z$.

\begin{prop} Гомоморфизм $\al_*\cln\pi_1(U)\to\pi_1(\Cbb^*)$ является изоморфизмом.
\end{prop}

\begin{proof} Достаточно доказать, что образ гомоморфизма $\al_*$ содержит образующий элемент группы $\pi_1(\Cbb^*)\cong\Z$.

Фиксируем вещественные числа $a$ и~$b$, такие что $a>b>0$ и~$a^2+b^2=1$, и~векторы $v_{1,2}:=\rbmat{a\\\pm b}\in S^3$. Для непрерывного отображения $[0;1]\times\bs{0;\frac{\pi}{2}}\to S^3$, сопоставляющего произвольным числам $t\in[0;1]$, $\be\in\bs{0;\frac{\pi}{2}}$ вектор $v_3(\be;t):=\rbmat{\cos\be\\\sin\be\cdot e^{2\pi it}}\in S^3$, имеем $v_3(0;t)=v_3(0;0)=\rbmat{1\\0}$ ($t\in[0;1]$).

Непрерывная функция $\ta_3\cln(S^3)^3\to\Cbb,\,(v_1;v_2;v_3)\to\br{\ta'_1(v_2;v_3;v_1)\cdot\ta'_1(v_3;v_1;v_2)}^2\bw\cdot\ol{\br{\ta_2(v_1;v_2;v_3)}}$ удовлетворяет соотношению
\eqn{\label{comp}
\fa p\in(S^3)^3\quad\quad(\al\circ\ta)(p)=\br{\ta'_1(p)}^2\ta_3(p).}

Заметим, что $\ta'_1\br{v_2;v_3(0;0);v_1}\bw=-b(3a^2-b^2)$, $\ta'_1\br{v_3(0;0);v_1;v_2}\bw=b(3a^2-b^2)$, $\ta_2\br{v_1;v_2;v_3(0;0)}\bw=4a^2b^6$,
откуда $\ta_3\br{v_1;v_2;v_3(0;0)}\ne0$. Открытое подмножество $\ta_3^{-1}(\Cbb^*)\subs(S^3)^3$, содержащее точку $\br{v_1;v_2;v_3(0;0)}$, содержит и~односвязное подмножество $\BC{\br{v_1;v_2;v_3(\be;t)}\cln\be\in[0;\be_0],\,t\in[0;1]}\subs(S^3)^3$ для некоторого $\be_0\in\br{0;\frac{\pi}{2}}$. Следовательно, петля $\ga\cln[0;1]\to(S^3)^3,\,t\to\br{v_1;v_2;v_3(\be_0;t)}$ стягиваема в~$\ta_3^{-1}(\Cbb^*)$, а~значит, $(\ta_3\circ\ga)\br{[0;1]}\subs\Cbb^*$ и~$[\ta_3\circ\ga]=0\in\pi_1(\Cbb^*)\cong\Z$. Кроме того, $(\ta'_1\circ\ga)\br{[0;1]}\subs\Cbb^*$ и~$[\ta'_1\circ\ga]=1\in\pi_1(\Cbb^*)\cong\Z$, поскольку
\ml{
\ta'_1\br{\ga(t)}=
(a\cos\be_0+b\sin\be_0\cdot e^{-2\pi it})(-b\cos\be_0-a\sin\be_0\cdot e^{2\pi it})+\\
+(a\cos\be_0-b\sin\be_0\cdot e^{-2\pi it})( b\cos\be_0-a\sin\be_0\cdot e^{2\pi it})
=-2\cos\be_0\sin\be_0\br{a^2e^{2\pi it}+b^2e^{-2\pi it}}}
и $\bm{a^2e^{2\pi it}}>\bm{b^2e^{-2\pi it}}$ для $t\in[0;1]$.

Согласно~\eqref{comp}, $(\al\circ\ta\circ\ga)\br{[0;1]}\subs\Cbb^*$ и~$[\al\circ\ta\circ\ga]=2[\ta'_1\circ\ga]+[\ta_3\circ\ga]=2\in\pi_1(\Cbb^*)\cong\Z$. Другими словами, петля $\ta\circ\ga\cln[0;1]\to S^3$ целиком содержится в~подмножестве $U\subs S^3$, а~образ элемента $[\ta\circ\ga]\in\pi_1(U)$ при гомоморфизме $\al_*$ равен $2\in\pi_1(\Cbb^*)\cong\Z$.

Пусть $g:=\diag(i;-i)\in\SU_2$. Для любого $t\in\bs{0;\frac{1}{2}}$ имеем $gv_1=iv_2$, $gv_2=iv_1$, $gv_3(\be_0;t)=iv_3\br{\be_0;t+\frac{1}{2}}$, и, в~силу~\eqref{fibta}, $\ta\Br{\ga\br{t+\frac{1}{2}}}=\ta\br{\ga(t)}$. Значит, $[\ta\circ\ga]=2c$, где $c\in\pi_1(U)$. Элемент~$c$ является искомым\: $2\al_*(c)=\al_*(2c)=\al_*[\ta\circ\ga]\bw=2\in\pi_1(\Cbb^*)\cong\Z$, $\al_*(c)=1\in\pi_1(\Cbb^*)\cong\Z$.
\end{proof}

Непрерывное отображение $U\to\Cbb^*,\,y\to y_3+y_4i$ также индуцирует изоморфизм $\pi_1(U)\to\pi_1(\Cbb^*)$. Можно считать, что он совпадает с~$\al_*$, сделав при необходимости замену координат $(y_1;y_2;y_3;y_4)\to(y_1;y_2;y_3;-y_4)$\: группы $\pi_1(U)$ и~$\pi_1(\Cbb^*)$ изоморфны группе~$\Z$, имеющей только два автоморфизма.

Существует непрерывное отображение $s\cln U\to\R$, удовлетворяющее равенству $\Arg\al(y)=s(y)+\Arg(y_3+y_4i)$ для всякого $y\in U$.

Перейдём к~построению гомеоморфизма пространства~$K$ и~сферы~$S^4$. Последнюю будет удобно понимать как единичную сферу
$\bc{(z_0;z_1;z_2)\in\R\oplus\Cbb^2\cln z_0^2+|z_1|^2+|z_2|^2=1}$ евклидова пространства $\R\oplus\Cbb^2$.

Отображение $\ph_0\cln\Cbb^2\to\R^3$, заданное формулой
\equ{
\ph_0(z_1;z_2):=
\case{
\bbr{\frac{|z_2|^2-|z_1|^2}{\sqrt{|z_1|^2+|z_2|^2}};
\frac{2\Rea(z_1^2\ol{z_2})}{|z_1|\sqrt{|z_1|^2+|z_2|^2}};\frac{2\Img(z_1^2\ol{z_2})}{|z_1|\sqrt{|z_1|^2+|z_2|^2}}},
&z_1\ne0;\\
(|z_2|;0;0),&z_1=0,}}
непрерывно, сюръективно, сохраняет длины векторов соответствующих евклидовых пространств, а~его слои суть орбиты действия $\T\cln\Cbb^2,\,\la\cln(z_1;z_2)\to(\la z_1;\la^2z_2)$. Прообразы подмножеств $\R_{\ge0}\times\bc{(0;0)}$ и~$\R_{\le0}\times\bc{(0;0)}$ пространства~$\R^3$ под действием~$\ph_0$ совпадают с~подмножествами $\{0\}\times\Cbb\subs\Cbb^2$ и~$\Cbb\times\{0\}\subs\Cbb^2$ соответственно. Следовательно, отображение $\ph\cln S^4\to\R^4\cong\R\oplus\R^3,\,(z_0;z_1;z_2)\to\br{z_0;\ph_0(z_1;z_2)}$ непрерывно, его образом является единичная сфера~$S^3$ пространства~$\R^4$, а~слоями\т орбиты действия $\T\cln S^4,\,\la\cln(z_0;z_1;z_2)\to(z_0;\la z_1;\la^2z_2)$\:
\eqn{\label{fibr}
\fa z=(z_0;z_1;z_2)\in S^4\quad\quad\ph^{-1}\br{\ph(z)}=\bc{(z_0;\la z_1;\la^2z_2)\cln\la\in\T}.}
Кроме того, $\ph^{-1}(I_j)=\bc{(z_0;z_1;z_2)\in S^4\cln z_j=0}$ и~$\ph^{-1}(U_j)=\bc{(z_0;z_1;z_2)\in S^4\cln z_j\ne0}$ ($j=1,2$).

Существует непрерывная функция $s_0\cln[0;1]\to\R$, равная нулю в~некоторой (односторонней) окрестности точки~$1$ и~удовлетворяющая равенству $s_0(t)+s_0(1-t)=1$ для всех $t\in[0;1]$. Пусть $j=1,2$, а~$s_j\cln\ph^{-1}(U_j)\to\R$\т функция, принимающая в~точке $z=(z_0;z_1;z_2)\in\ph^{-1}(U_j)$ значение
\equ{
s_j(z):=
\case{
\frac{(-1)^jj}{2}\cdot s\br{\ph(z)}\cdot s_0\Br{\frac{|z_j|^2}{|z_1|^2+|z_2|^2}},&z\in\ph^{-1}(U);\\
0,&z\in\ph^{-1}(U_j\sm U).}}
Функция~$s_j$ непрерывна\: она равна нулю в~некоторой окрестности подмножества $\ph^{-1}(U_j\sm U)$, поскольку при $z\in\ph^{-1}(U_j\sm U)$ число $\frac{|z_j|^2}{|z_1|^2+|z_2|^2}$ равно~$1$. Для точки $z=(z_0;z_1;z_2)\in S^4$ имеем $\al_j\br{\ph(z)}\bw=0\bw\Lra\ph(z)\in I_j\Lra z_j=0$, поэтому
функция
\equ{
\ph_j\cln\;\;S^4\to\Cbb,\;\;z=(z_0;z_1;z_2)\;\;\to\;\;
\case{
\al_j\br{\ph(z)}\cdot\frac{z_j}{|z_j|}\cdot\exp\br{-is_j(z)},&z_j\ne0;\\
0,&z_j=0}}
удовлетворяет соотношениям $|\ph_j(z)|=\al_j\br{\ph(z)}$, $\ph_j(z)\bw=0\bw\Lra z_j=0$ ($z=(z_0;z_1;z_2)\in S^4$), а~значит, непрерывна.

Если $z=(z_0;z_1;z_2)\in\ph^{-1}(U)$ и~$y:=\ph(z)\in S^3$, то $s_0\Br{\frac{|z_1|^2}{|z_1|^2+|z_2|^2}}+s_0\Br{\frac{|z_2|^2}{|z_1|^2+|z_2|^2}}=1$, $s_2(z)-2s_1(z)=s(y)$, комплексное число $y_3+y_4i$ равно $\frac{2z_1^2\ol{z_2}}{|z_1|\sqrt{|z_1|^2+|z_2|^2}}$, и
\eqn{\label{args}
2\Arg\br{\ph_1(z)}-\Arg\br{\ph_2(z)}=\Arg\br{z_1^2\ol{z_2}}+\br{s_2(z)-2s_1(z)}=\Arg(y_3+y_4i)+s(y)=\Arg\al(y).}

Рассмотрим непрерывное отображение $\wt{\ph}\cln S^4\to S^3\times\Cbb^2,\,z\to\br{\ph(z);\ph_1(z);\ph_2(z)}$.

\begin{lemma} Отображение $\wt{\ph}$ инъективно, а~его образ есть~$K$.
\end{lemma}

\begin{proof} Положим $K':=\wt{\ph}(S^4)\subs S^3\times\Cbb^2$. Из~\eqref{fibr} следует, что\:
\begin{gather}
\fa j=1,2\;\fa z=(z_0;z_1;z_2)\in S^4\;\fa\la\in\T\quad\quad\ph_j(z_0;\la z_1;\la^2z_2)=\la^j\ph_j(z);\notag\\
\fa z\in S^4\quad\quad K'\cap\Br{\bc{\ph(z)}\times\Cbb^2}=\BC{\br{\ph(z);\la\ph_1(z);\la^2\ph_2(z)}\cln\la\in\T}.\label{orb'}
\end{gather}
Значит, любые две точки одного слоя отображения $\wt{\ph}$ имеют вид $z=(z_0;z_1;z_2)\in S^4$ и~$(z_0;\la z_1;\la^2z_2)\in S^4$, $\la\in\T$, причём $\la^j\ph_j(z)=\ph_j(z)$ ($j=1,2$). Но $\ph_j(z)=0\bw\Lra z_j=0$, откуда $\la^jz_j=z_j$ ($j=1,2$), то есть $(z_0;\la z_1;\la^2z_2)\bw=z$. Тем самым доказана инъективность отображения $\wt{\ph}$.

Пусть $y\in S^3$, $K_y:=\bc{w\in\Cbb^2\cln(y;w)\in K}\subs\Cbb^2$, $K'_y:=\bc{w\bw\in\Cbb^2\cln(y;w)\bw\in K'}\bw\subs\Cbb^2$. Докажем, что $K_y=K'_y$.

Найдутся точки $p\in(S^3)^3$ и~$z\in S^4$, такие что $\ta(p)=\ph(z)=y$. В~силу \eqref{orb} и~\eqref{orb'}, $K_y$ и~$K'_y$\т орбиты действия $\T\cln\Cbb^2,\,\la\cln(w_1;w_2)\to(\la w_1;\la^2w_2)$, содержащие точки $\br{\ta_1(p);\ta_2(p)}$ и~$\br{\ph_1(z);\ph_2(z)}$ соответственно. Заметим, что $\bm{\ta_j(p)}=\al_j(y)=\bm{\ph_j(z)}$ ($j=1,2$), а~если $\al_1(y),\al_2(y)\ne0$, то $z\in\ph^{-1}(U)$ и, согласно~\eqref{args}, аргумент числа $\br{\ph_1(z)}^2\ol{\br{\ph_2(z)}}$ равен аргументу числа $\al(y)=\br{\ta_1(p)}^2\ol{\br{\ta_2(p)}}$. Поэтому $K_y=K'_y$.

Ввиду произвольности точки $y\in S^3$, подмножество $K\subs S^3\times\Cbb^2$ совпадает с~$K'\bw=\wt{\ph}(S^4)$.
\end{proof}

Теперь, когда построен гомеоморфизм $\wt{\ph}\cln S^4\to K$, теорема~\ref{n8} полностью доказана.

\subsection{Доказательство теоремы~\ref{n53}}

Можно считать, что $V_0=0$, а~представление $G\cln V$ есть прямая сумма неприводимых представлений $G\cln V_1$ и~$G\cln V_2$
размерностей $5$ и~$3$ соответственно. Эти представления группы $G\cong\SO_3$ можно отождествить соответственно с~её действием на пространстве $\bc{A\in\End(\R^3)\cln A=A^T,\,\tr A=0}$ матричными сопряжениями и~с~тавтологическим представлением $\SO_3\cln\R^3$.

Положим $D:=\BC{\diag(\la_1,\la_2\la_3)\in\End(\R^3)\cln\suml{i=1}{3}\la_i=0,\,\la_1\bw\ge\la_2\bw\ge\la_3}\subs V_1$. Далее, для оператора $A\in D$ обозначим через~$Z_A$ подгруппу $\{g\in\SO_3\cln gA=Ag\}\subs\SO_3$. Действие $Z_A\cln\R^3$ задаёт отношение эквивалентности~$\tau_A$ на пространстве~$\R^3$.

Каждая орбита действия $G\cln V$ пересекает подмножество $D\times\R^3\subs V$, причём элементы $(A;v)$ и~$(A';v')$ ($A,A'\bw\in D$, $v,v'\bw\in\R^3$) лежат в~одной орбите тогда и~только тогда, когда $A=A'$ и $v\,\tau_A\,v'$. Отсюда каноническое отображение
$D\times\R^3\hra V\thra V/G$ сюръективно, а~его слои суть декартовы произведения одноэлементных подмножеств $\{A\}\subs D$ на классы соответствующих отношений эквивалентности~$\tau_A$ в~$\R^3$. Ограничивая действия $Z_A\cln\R^3$ ($A\in D$) на инвариантную единичную сферу~$S^2$, получаем, что фактор $V/G$ гомеоморфен факторпространству топологического пространства $D\times\R_{\ge0}\times S^2$ по отношению эквивалентности
\equ{
(A;r;v)\sim(A;r;v')\cln\quad(r=0)\vee(v\,\tau_A\,v')\quad(A\in D,\,r\in\R_{\ge0},\,v,v'\in S^2).}

Подмножество $D\subs\End(\R^3)$ содержит операторы $A_0:=\diag(1;1;-2)$ и~$A_1:=\diag(2;-1;-1)$, гомеоморфно замкнутой полуплоскости, а~его граница есть объединение лучей $I_i:=\R_{\ge0}A_i$ ($i=0,1$), пересекающихся по нулевому оператору. Следовательно, существует гомеоморфизм топологических пространств $D$ и
\equ{
\br{\R_{\ge0}\bw\times[0;1]}\hb{\bc{(0;t)\bw\sim(0;t')\cln t,t'\in[0;1]}},}
переводящий каждый из лучей $I_i\subs D$ ($i=0,1$) в~подмножество $\R_{\ge0}\times\{i\}$ (а~точку $0\in D$\т в~класс эквивалентности $\{0\}\times[0;1]\subs\R_{\ge0}\times[0;1]$). В дальнейшем мы будем отождествлять эти топологические пространства с~помощью указанного гомеоморфизма.

Всякая подгруппа вида $Z_A$, $A\in D$, содержит подгруппу~$H$ всех диагональных операторов группы~$\SO_3$. Более точно,
\eqn{\label{cent}
Z_A=\case{
\SO_3,&A=0;\\
Z_{A_i},&A\in I_i\sm\{0\}\quad(i=0,1);\\
H,&A\in\Int D.}}

Выясним, как устроен фактор $S^2/H$. Для этого заметим, что $H$ есть чётная подгруппа конечной группы отражений, соответствующей системе корней
$\{\pm e_i\cln i=1,2,3\}$, где $(e_1,e_2,e_3)$\т стандартный базис в~$\R^3$. Через~$C$ обозначим камеру Вейля
$\bc{(x_1;x_2;x_3)\bw\in\R^3\cln x_i\bw\ge0}\bw\subs\R^3$, а~через~$M$\т её пересечение со сферой~$S^2$. Отображение
$M\bw\times\{\pm1\}\bw\to S^2,\,\br{(x_1;x_2;x_3);j}\bw\to(jx_1;x_2;x_3)$ ($(x_1;x_2;x_3)\bw\in M$, $j=\pm1$) задаёт естественный гомеоморфизм~$\ph$ топологических пространств
\eqn{\label{whM}
\wh{M}:=\br{M\times\{\pm1\}}\hb{\bc{(v;1)\sim(v;-1)\cln v\in\pd M}}}
и~$S^2/H$. В~дальнейшем будем отождествлять их при помощи~$\ph$.

Пусть $A\in D$\т произвольный оператор. Каждая орбита действия группы $H\subs Z_A$ на сфере~$S^2$ целиком содержится в одном классе отношения эквивалентности~$\tau_A$, что позволяет ввести индуцированное отношение эквивалентности~$\wt{\tau_A}$ на факторе $S^2/H\cong\wh{M}$. В таком случае фактор $V/G$ гомеоморфен факторпространству топологического пространства $D\times\R_{\ge0}\times\wh{M}$ по отношению эквивалентности
\equ{
(A;r;w)\sim(A;r;w')\cln\quad(r=0)\vee(w\,\wt{\tau_A}\,w')\quad\br{A\in D,\,r\in\R_{\ge0},\,w,w'\in\wh{M}\,}.}

Положим $\tau_i:=\tau_{A_i}$, $\wt{\tau_i}:=\wt{\tau_{A_i}}$ ($i=0,1$). Согласно~\eqref{cent}, отношение эквивалентности~$\wt{\tau_A}$
при $A\in\Int D$ тривиально (все элементы попарно не эквивалентны), при $A\in I_i\sm\{0\}$ ($i=0,1$) совпадает с отношением $\wt{\tau_i}$,
а~при $A=0$ имеет единственный класс эквивалентности\т $\wh{M}$. Отсюда фактор $V/G$ гомеоморфен факторпространству декартова произведения $\R_{\ge0}^2\times[0;1]\times\wh{M}$ по отношению эквивалентности
\eqn{\label{equi}
\begin{aligned}
(0;r;t;w)&\sim(0;r;t';w')&&&&\br{r\in\R_{\ge0},\,t,t'\in[0;1],\,w,w'\in\wh{M}\,};\\
(s;0;t;w)&\sim(s;0;t;w') &&&&\br{s\in\R_{\ge0},\,t\in[0;1],\,w,w'\in\wh{M}\,};\\
(s;r;t;w)&\sim(s;r;t;w') &&&&\br{s,r\in\R_{\ge0},\,t=0,1,\,w,w'\in\wh{M},\,w\,\wt{\tau_t}\,w'}.
\end{aligned}}

Установим связь между отношениями эквивалентности $\wt{\tau_0}$ и~$\wt{\tau_1}$.

Рассмотрим оператор $g\in\SO_3$, переводящий базисные векторы $e_1,e_2,e_3\in\R^3$ базиса в~векторы $e_3,e_1,e_2$ соответственно. Легко видеть, что $A_0=-gA_1g^{-1}$, $Z_{A_0}=gZ_{A_1}g^{-1}$, $\tau_0=g\tau_1$ (т.\,е. $v\,\tau_1\,v'\bw\Lra(gv)\,\tau_0\,(gv')$). Кроме того, оператор~$g$ нормализует подгруппу $H\subs\SO_3$ и~потому индуцирует гомеоморфизм $\wt{g}\cln S^2/H\to S^2/H$, причём $\wt{\tau_0}=\wt{g}\wt{\tau_1}$.

\begin{prop} В классе гомеоморфизмов $\wh{M}\to\wh{M}$ гомеоморфизмы $\wt{g}$ и $\id_{\wh{M}}$ гомотопны, то есть существует семейство гомеоморфизмов $\wt{g}_t\cln\wh{M}\bw\to\wh{M}$ \ter{$t\in[0;1]$}, такое что $\wt{g}_0\equiv\id_{\wh{M}}$, $\wt{g}_1\equiv\wt{g}$, а~отображение
$[0;1]\times\wh{M}\bw\to\wh{M},\,(t;w)\bw\to\wt{g}_t(w)$ непрерывно.
\end{prop}

\begin{proof} Ясно, что $gC=C$ и~$gM=M$. Гомеоморфизм $\wt{g}\cln\wh{M}\to\wh{M}$ индуцируется гомеоморфизмом $g|_M\cln M\bw\to M$ с~учётом определения~\eqref{whM}. В классе гомеоморфизмов $M\to M$ гомеоморфизм $g|_M$, будучи поворотом сферического треугольника~$M$, гомотопен $\id_M$, что немедленно влечёт требуемое.
\end{proof}

Гомеоморфизм пространства $\R_{\ge0}^2\times[0;1]\times\wh{M}$ в~себя, действующий по формуле $(s;r;t;w)\to\br{s;r;t;\wt{g}_t(w)}$, превращает отношение эквивалентности~\eqref{equi} в~отношение эквивалентности
\eqn{\label{equii}
\begin{aligned}
(0;r;t;w)&\sim(0;r;t';w')&&&&\br{r\in\R_{\ge0},\,t,t'\in[0;1],\,w,w'\in\wh{M}\,};\\
(s;0;t;w)&\sim(s;0;t;w') &&&&\br{s\in\R_{\ge0},\,t\in[0;1],\,w,w'\in\wh{M}\,};\\
(s;r;t;w)&\sim(s;r;t;w') &&&&\br{s,r\in\R_{\ge0},\,t=0,1,\,w,w'\in\wh{M},\,w\,\wt{\tau_0}\,w'}.
\end{aligned}}
Следовательно, фактор $V/G$ гомеоморфен факторпространству декартова произведения $\R_{\ge0}^2\bw\times[0;1]\bw\times\wh{M}$ по отношению эквивалентности~\eqref{equii}.

При помощи сферической замены координат
$x_1=\sin\br{\frac{\pi}{2}u_1}\cos\br{\frac{\pi}{2}u_2}$, $x_2=\sin\br{\frac{\pi}{2}u_1}\sin\br{\frac{\pi}{2}u_2}$, $x_3=\cos\br{\frac{\pi}{2}u_1}$ получаем, что $M\cong[0;1]^2\hb{\bc{(0;u_2)\bw\sim(0;u'_2)\cln u_2,u'_2\bw\in[0;1]}}$, а~пространство~$\wh{M}$ гомеоморфно факторпространству декартова произведения $[0;1]^2\times\{\pm1\}$ по отношению эквивалентности
\equ{
\begin{aligned}
(0;u_2;j)  &\sim(0;u'_2;j')  &&&&\br{u_2,u'_2\in[0;1],\,j,j'=\pm1};\\
(1;u_2;1)  &\sim(1;u_2;-1)   &&&&\br{u_2\in[0;1]};\\
(u_1;u_2;1)&\sim(u_1;u_2;-1) &&&&\br{u_1\in[0;1],\,u_2=0,1}.
\end{aligned}}

Подгруппа $Z_{A_0}\subs\SO_3$ порождена вращениями вокруг прямой $\R e_3$ на всевозможные углы, а~также оператором $\diag(1;-1;-1)$. Значит, отношение эквивалентности~$\wt{\tau_0}$ на множестве $\wh{M}$ есть не что иное как
$\bc{(u_1;u_2;j)\bw\sim(u_1;u'_2;j')\cln u_1,u_2,u'_2\bw\in[0;1],\,j,j'\bw=\pm1}$,
а~фактор $V/G$ гомеоморфен факторпространству декартова произведения $\R_{\ge0}^2\bw\times[0;1]^3\bw\times\{\pm1\}$ по отношению эквивалентности
\eqn{\label{equz}
\begin{aligned}
(0;r;t;u_1;u_2;j)&\sim(0;r;t';u'_1;u'_2;j')&&&&\br{r\in\R_{\ge0},\,t,t',u_i,u'_i\in[0;1],\,j,j'=\pm1};\\
(s;0;t;u_1;u_2;j)&\sim(s;0;t;u'_1;u'_2;j') &&&&\br{s\in\R_{\ge0},\,t,u_i,u'_i\in[0;1],\,j,j'=\pm1};\\
(s;r;t;u_1;u_2;j)&\sim(s;r;t;u_1;u'_2;j')  &&&&\br{s,r\in\R_{\ge0},\,t=0,1,\,u_1,u_2,u'_2\in[0;1],\,j,j'=\pm1};\\
(s;r;t;0;u_2;j)  &\sim(s;r;t;0;u'_2;j')    &&&&\br{s,r\in\R_{\ge0},\,t,u_2,u'_2\in[0;1],\,j,j'=\pm1};\\
(s;r;t;1;u_2;1)  &\sim(s;r;t;1;u_2;-1)     &&&&\br{s,r\in\R_{\ge0},\,t,u_2\in[0;1]};\\
(s;r;t;u_1;u_2;1)&\sim(s;r;t;u_1;u_2;-1)   &&&&\br{s,r\in\R_{\ge0},\,t,u_1\in[0;1],\,u_2=0,1}.
\end{aligned}}

Образ отображения
\equ{
\pi\cln\R_{\ge0}^2\bw\times[0;1]^3\bw\times\{\pm1\}\to\R^5,\,(s;r;t;u_1;u_2;j)\to\br{s;\,r;\,st;\,sru_1;\,s^3rt(1-t)u_1u_2}}
совпадает с подмножеством
\equ{
K:=\bc{(s;r;t;u_1;u_2)\in\R^5\cln s,r,t,u_1,u_2\ge0,\,t\le s,\,u_1\le sr,\,u_2\le t(s-t)u_1}\subs\R^5.}
Функция $\al\cln\R^5\to\R,\,(s;r;t;u_1;u_2)\to(sr-u_1)u_2\br{t(s-t)u_1-u_2}$ равна нулю на $\pd K$ и~положительна в~$\Int K$.
Поэтому образ отображения
\equ{
\wt{\pi}\cln\R_{\ge0}^2\bw\times[0;1]^3\bw\times\{\pm1\}\to\R^6=\R^5\oplus\R,\,p=(s;r;t;u_1;u_2;j)\to\Br{\pi(p);j\al\br{\pi(p)}}}
совпадает с подмножеством
\equ{
\bc{(p;y)\in K\times\R\cln|y|=\al(p)}\subs K\times\R\subs\R^5\oplus\R=\R^6,}
гомеоморфным $\br{K\times\{\pm1\}}\hb{\bc{(x;1)\sim(x;-1)\cln x\in\pd K}}$. Его слои суть классы отношения эквивалентности~\eqref{equz}:
\equ{
\wt{\pi}(s;r;t;u_1;u_2;j)=\br{s;\,r;\,st;\,sru_1;\,s^3rt(1-t)u_1u_2;js^7r^3t^2(1-t)^2u_1^2(1-u_1)u_2(1-u_2)}.}

Существует гомеоморфизм из пятимерного замкнутого полупространства в~подмножество $K\subs\R^3$, переводящий граничную гиперплоскость в~$\pd K$.

\end{document}